\documentclass[10pt,reqno]{amsart}
\usepackage{amssymb,mathrsfs,graphicx,amsmath,mathtools}
\usepackage{ifthen}
\usepackage{ifthen,latexsym,float,colortbl}
\usepackage[margin=1in]{geometry}
\usepackage{caption}
\usepackage{sidecap}
\usepackage{rotating}

\provideboolean{shownotes} 
\setboolean{shownotes}{true} 
\newcommand{\margnote}[1]{
\ifthenelse{\boolean{shownotes}}%
{\marginpar{\raggedright\tiny\texttt{#1}}}%
{}%
}
\newcommand{\hole}[1]{
\ifthenelse{\boolean{shownotes}}%
{\begin{center} \fbox{ \rule {.25cm}{0cm} \rule[-.1cm]{0cm}{.4cm}
\parbox{.85\textwidth}{\begin{center} \texttt{#1}\end{center}} \rule
{.25cm}{0cm}}\end{center}} {} }

\title[Large friction limit of pressureless Euler equations with nonlocal forces]{Large friction limit of pressureless Euler equations with nonlocal forces}

\author[Choi]{Young-Pil Choi}
\address[Young-Pil Choi]{\newline Department of Mathematics \newline
Yonsei University, 50 Yonsei-Ro, Seodaemun-Gu, Seoul 03722, Republic of Korea}
\email{ypchoi@yonsei.ac.kr}

\numberwithin{equation}{section}

\newtheorem{theorem}{Theorem}[section]
\newtheorem{lemma}{Lemma}[section]

\newtheorem{proposition}{Proposition}[section]
\newtheorem{remark}{Remark}[section]
\newtheorem{definition}{Definition}[section]
\DeclareMathOperator*{\esssup}{ess\,sup}

\newcommand{\R}{\mathbb R}

\newcommand{\om}{\Omega}

\newcommand{\T}{\mathbb T}

\newcommand{\mc}{\mathcal C}

\newcommand{\bq}{\begin{equation}}
\newcommand{\eq}{\end{equation}}
\newcommand{\e}{\varepsilon}
\newcommand{\lt}{\left}
\newcommand{\rt}{\right}

\newcommand{\pa}{\partial}

\newcommand{\mt}{\mathcal{T}}
\newcommand{\W}{\mathcal{W}}
\newcommand{\intt}{\int_{\om}}
\newcommand{\inttt}{\iint_{\om \times \om}}

\begin{document}
\allowdisplaybreaks


\subjclass[]{}

\begin{abstract} We rigorously show a large friction limit of hydrodynamic models with alignment, attractive, and repulsive effects. More precisely, we consider pressureless Euler equations with nonlocal forces and provide a quantitative estimate of large friction limit to a continuity equation with nonlocal velocity fields, which is often called an aggregation equation. Our main strategy relies on the relative entropy argument combined with the estimate of $p$-Wasserstein distance between densities.

\end{abstract}

\maketitle \centerline{\date}


%
%
%
%
\section{Introduction}
In this paper, we are interested in a large friction limit of pressureless Euler equations with nonlocal forces, referred to as Euler-Alignment models \cite{CCTT16}, in the domain $\om$, which is either $\T^d$ or $\R^d$, with $d\geq 1$. Let $\rho = \rho(x,t)$ and $u = u(x,t)$ be the density and velocity of the flow at $(x,t) \in \om \times \R_+$, respectively. Then our main system is given by 
\begin{align}\label{main_eq}
\begin{aligned}
&\pa_t \rho + \nabla_x \cdot (\rho u) = 0, \quad (x,t) \in \om \times \R_+,\cr
&\e\pa_t (\rho u) + \e \nabla_x \cdot (\rho u \otimes u) = -\gamma \rho u - (\nabla_x W \star \rho)\rho + \rho \intt \phi(x-y)(u(y) - u(x))\rho(y)\,dy.
\end{aligned}
\end{align}
Here $\gamma > 0$ is the strength of linear damping, $W: \om \to \R$ denotes the interaction potential, $\phi:\om \to \R_+$ represents a communication weight function. Throughout this paper, we assume that $\phi$ and $W$ satisfy $\phi(-x) = \phi(x)$ and $W(-x) = W(x)$ for $x \in \om$, respectively. 

The macroscopic model \eqref{main_eq} can be derived from Newton-type equations, which is a microscopic model, via a kinetic formulation. Consider a system of $N$ particles whose state can be defined by positions $x_i(t)$ and velocities $v_i(t)$, respectively, at time $t>0$. The evolution of this system is governed by the following system of ordinary differential equations:
\begin{align}\label{par_sys}
\begin{aligned}
\frac{d x_{i}(t)}{dt} &= v_i(t), \quad i=1,\dots, N, \quad t > 0,\cr
\e\frac{d v_{i}(t)}{dt} &= -\gamma v_i(t) - \frac1N\sum_{j=1}^N \nabla_x W(x_i(t)-x_j(t)) +  \frac1N\sum_{j=1}^N \phi(x_i(t) - x_j(t))(v_j(t) - v_i(t)).
\end{aligned}
\end{align}
The second term on the right hand side of the above differential equations for $v_i$ represents attractive/repulsive forces, and the third serves as a nonlocal velocity alignment force, see \cite{CCP17} for more discussion. If we ignore the linear damping and the interactions between particles through the potential function $W$, i.e., $\gamma = 0$ and $W \equiv 0$, then the particle system \eqref{par_sys} becomes the celebrated Cucker--Smale model \cite{CS07,HL09,HT08} for flocking behaviors. We refer to \cite{CHL17} for a general introduction to the Cucker--Smale model and its variants.

As the number of particles $N$ goes to infinity, we can derive a kinetic equation by means of mean-field limits or BBGKY hierarchies \cite{CCHS19, HL09, HT08}. More precisely, let $f = f(x,v,t)$ be the one particle distribution function. Then $f$ solves the following Vlasov-type equation, which is a mesoscopic model:
\bq\label{eq_kin}
\pa_t f + v \cdot \nabla_x f + \nabla_v \cdot (F[f]f) = 0,
\eq
where the force $F[f] = F[f](x,v,t)$ is given by
\[
F[f](x,v,t) = \frac1\e\lt( - \gamma v -  (\nabla_x W \star \rho)(x,t) + \iint_{\om \times \R^d} \phi(x-y)(w-v)f(y,w,t)\,dydw \rt).
\]
Here $\rho$ denotes the local particle density, i.e., 
\[
\rho(x,t) = \int_{\R^d} f(x,v,t)\,dv,
\]
and $\star$ stands for the convolution operator in spatial variable.

Our main hydrodynamic equations \eqref{main_eq} can be obtained by taking care of local moments in velocity on the above kinetic equation \eqref{eq_kin} together with a mono-kinetic distribution for $f$. Indeed, if we define a local particle velocity $u = u(x,t)$ as
\[
u(x,t) = \int_{\R^d} v f(x,v,t)\,dv \bigg / \int_{\R^d} f(x,v,t)\,dv,
\]
then we can easily check from \eqref{eq_kin} that the local density $\rho$ and velocity $u$ satisfy
\begin{align}\label{eq_hydro}
\begin{aligned}
&\pa_t \rho + \nabla_x \cdot (\rho u) =0,\cr
&\e\pa_t (\rho u) + \e\nabla_x \cdot (\rho u \otimes u) + \e\nabla_x \cdot \lt( \int_{\R^d} (v-u) \otimes (v-u) f(x,v,t)\,dv \rt)\cr
&\quad = - \gamma \rho u -(\nabla_x W \star \rho) \rho + \rho \intt \phi(x-y)(u(y)-u(x))\rho(y)\,dy.
\end{aligned}
\end{align}
In order to close the above system, we assume 
\[
f (x,v,t) \simeq \rho(x,t) \delta_{u(x,t)}(v),
\] 
where $\delta$ denotes the Dirac measure. Then the system \eqref{eq_hydro} becomes our main pressureless Euler-type system \eqref{main_eq}. Note that the closure based on the mono-kinetic distribution can be justified by considering an additional force term, for an instance a local velocity alignment force $\nabla_v \cdot ((u-v)f)$, with a singular parameter \cite{CCpre, Cpre, FK19}.

There are several works on the pressureless Euler equations with nonlocal forces. Without the linear damping and the interaction potential $W$ in \eqref{main_eq}, the global existence and the long time behavior of strong solutions are obtained in \cite{HKK15}, see also \cite{CH19} for the normalized communication weight case. In one dimensional case, sharp critical thresholds between a supercritical region with finite-time blow-up and a subcritical region with global-in-time regularity of classical solutions are discussed in \cite{CCTT16}, see also \cite{CCZ16,TT14}. Including the pressure term, we also refer to \cite{CFGS17,CPW20,C19} for the existence of weak and strong solutions. 

In the current work, we are interested in the behavior of solutions $(\rho^\e, \rho^\e u^\e)$ to the system \eqref{main_eq} as $\e \to 0$. At the formal level, it is expected to have that the solutions $(\rho^\e, \rho^\e u^\e)$ converge toward solutions $(\rho, \rho u)$ which solve the following continuity equation with nonlocal velocity fields, which is often called an aggregation equation: 
\bq\label{main_eq2}
\pa_t \rho + \nabla_x \cdot (\rho u) = 0, \quad (x,t) \in \om \times \R_+,
\eq
where
\bq\label{main_eq22}
\rho u = - \frac\rho\gamma(\nabla_x W \star \rho) + \frac\rho\gamma \intt \phi(x-y)(u(y) - u(x))\rho(y)\,dy.
\eq

There are some studies on the large friction limit from Euler-type equations to the aggregation-diffusion equation or Keller--Segel equation \cite{CCT19, CPW20, CG07, LT13, LT17}. Here the main mathematical tool is based on the relative entropy method proposed in \cite{Daf79} to study the weak-strong uniqueness principle. It is worth noticing that in these previous works the pressure term in the Euler equations plays a crucial role in analyzing the large friction limit since the nonlocal interaction terms can be dominated by the relative pressure. However, our main system is the pressureless Euler-type system, thus it is not clear how to estimate the nonlocal interaction forces. To the best of author's knowledge, the large friction limit of pressureless Euler equations with nonlocal interaction forces has not been studied so far. In the current work, we consider two different types of interaction potentials, regular and Coulomb ones. Here the regular interaction potential means that $W$ is bounded and Lipschitz continuous, and the Coulomb one represents that $W$ is given as the fundamental solution of Laplace's equation, i.e., $-\Delta_x W = \delta_0$, where $\delta_0$ denotes the Dirac measure. For the Coulomb interaction potential case, motivated from \cite{CFGS17,LT17}, we use the particular structure of the Poisson equation carefully to estimate the nonlocal interaction forces. For the regular case, even though the interaction potential has a good regularity, it is not clear how to have the benefit from that. In general, this can be controlled by the relative pressure \cite{CCK16, CPW20, KMT15} under suitable regularity assumptions for the interaction potential $W$. In order to resolve the difficulty caused by the absence of the pressure, inspired by recent works \cite{CCpre,Cpre,CYpre,FK19}, we use the $p$-Wasserstein distance with $p \in [1,\infty)$ which is defined by
\[
d_p(\mu, \nu) := \inf_{\gamma \in \Gamma(\mu,\nu)} \lt(\inttt |x-y|^p\, \gamma(dx,dy)\rt)^{1/p},
\]
and for $p =\infty$, which is the limiting case as $p \to \infty$, the $\infty$-Wasserstein distance is defined by
\[
d_\infty(\mu, \nu):=\inf_{\gamma \in \Gamma(\mu,\nu)} \esssup_{(x,y) \in supp(\gamma)}|x-y|
\]
for $\mu,\nu \in \mathcal{P}_p(\om)$, where $\Gamma(\mu,\nu)$ is the set of all probability measures on $\om \times \om$ with first and second marginals $\mu$ and $\nu$, respectively, i.e.,
\[
\iint_{\om \times \om} (\varphi(x) + \psi(y))\gamma(dx,dy) = \intt \varphi(x)\,\mu(dx) +  \intt \psi(y)\,\nu(dy)
\]
for each $\varphi, \psi \in \mc(\om)$. Here $\mathcal{P}_p(\om)$ is the set of probability measures in $\om$ with $p$-th moment bounded. Note that $\mathcal{P}_p(\om)$ is a complete metric space endowed with the $p$-Wasserstein distance, and in particular, $1$-Wasserstein distance is equivalent to the bounded Lipschitz distance in the metric space $\mathcal{P}_1(\om)$. We refer to \cite{AGS08,Vill03} for detailed discussions of various topics related to the Wasserstein distance.

We employ the Wasserstein distance to estimate the term related to the nonlocal interaction force under the regularity assumptions on the interaction potential function $W$ and the communication weight function $\phi$. We also show that the $p$-Wasserstein distance with $p\in[1,2]$ can be also bounded from above by the relative entropy functional for the pressureless Euler equations, see Section \ref{sec_wasser} for more details. 

\begin{remark} 
The large friction limit of the particle system \eqref{par_sys} can be considered. By Tikhonov theorem \cite{Tik52}, under suitable assumptions on the interaction potential $W$, the communication weight function $\phi$, and the initial data, we can derive from \eqref{par_sys} the following system  of ordinary differential equations as $\e \to 0$:
$$\begin{aligned}
\frac{dx_i(t)}{dt} &= v_i(t), \quad i=1,\dots, N, \quad t > 0,\cr
\lt(\gamma + \frac1N\sum_{j=1}^N \phi(x_i(t) - x_j(t)) \rt)v_i(t) &= -\frac1N\sum_{j=1}^N \nabla_x W (x_i(t) - x_j(t)) + \frac1N\sum_{j=1}^N \phi(x_i(t) - x_j(t))v_j(t).
\end{aligned}$$
Note that if we send $N \to \infty$ in the above system, at the formal level, we can derive the continuity equation \eqref{main_eq2}.
\end{remark}

\begin{remark} The large friction limit of the kinetic equation \eqref{eq_kin} with $\phi \equiv 0$ is studied in \cite{FS15,Jab00} by using PDE analysis and the method of characteristics. These results are also extended to the case with the velocity alignment force \cite{FST16}, i.e., the kinetic equation \eqref{eq_kin} with $\gamma = 0$. More recently, the quantitative estimate for the large friction limit is also discussed in \cite{CCpre}. 
\end{remark}

We now introduce several simplified notations that will be used throughout the paper. For a function $f(x)$, $\|f\|_{L^p}$ denotes the usual $L^p(\om)$-norm. We also denote by $C$ a generic positive constant which may differ from line to line. We also drop $x$-dependence of differential operators, i.e., $\nabla f = \nabla_x f$ and $\Delta f = \Delta_x f$. For any nonnegative integer $k$ and $p \in [1,\infty]$, $\W^{k,p}:= \W^{k,p}(\om)$ stands for the $k$-th order $L^p$ Sobolev space. Furthermore, we set $\mc^k(I;\mathcal{B})$ be the set of $k$-times continuously differentiable functions from an interval $I$ to a Banach space $\mathcal{B}$.

The rest of this paper is organized as follows. In Section \ref{sec_res}, we introduce definitions of solutions to the equations \eqref{main_eq} and \eqref{main_eq2}--\eqref{main_eq22}, and we also present our main results on the large friction limits $\e \to 0$ in \eqref{main_eq}. In Section \ref{sec_wasser}, we develop a general theory for the relation between $p$-Wasserstein distance and the relative entropy-type functional. Section \ref{sec_pf} is devoted to provide the details of the large friction limit when the interaction potential $W$ is regular. As mentioned above, in this case, we combine the relative entropy functional and the $p$-Wasserstein distance to have the quantitative estimate between two solutions to the equations \eqref{main_eq} and \eqref{main_eq2}--\eqref{main_eq22}. Finally, in Section \ref{sec_poisson}, we present the details of proof for the Coulomb interaction potential case. 

%
%
%
%
%
%
%
\section{Main results}\label{sec_res}
In this section, we present our main results on the large friction limit from the pressureless Euler equations with nonlocal forces \eqref{main_eq} to the continuity equation \eqref{main_eq2}. For this, we first introduce some notion of solutions to the equations \eqref{main_eq} and \eqref{main_eq2} \cite{CPW20,LT13,LT17}.  

\begin{definition}\label{def_weak}For a given $T \in (0,\infty)$, we say that $(\rho^\e,\rho^\e u^\e)$ is a weak solution to the system \eqref{main_eq} if the following conditions are satisfied:
\begin{itemize}
\item[(i)] $\rho^\e \in \mc([0,T);L^1_+(\om))$ and $\rho^\e |u^\e|^2 \in \mc([0,T);L^1(\om))$,
\item[(ii)] $(\rho^\e, \rho^\e u^\e)$ satisfies the system \eqref{main_eq} in the sense of distributions, and
\item[(iii)] $(\rho^\e, \rho^\e u^\e)$ satisfies the following weak formulation of energy equality:
\begin{align}\label{weak_form1}
\begin{aligned}
&-\frac\e2\int_0^T\intt \rho^\e|u^\e|^2 \psi'(t)\,dxdt + \gamma\int_0^T\intt \rho^\e|u^\e|^2 \psi(t)\,dxdt\cr
&\quad + \frac{1}{2}\int_0^T \inttt \phi(x-y)|u^\e(x) - u^\e(y)|^2 \rho^\e(x) \rho^\e(y)\psi(t)\,dxdydt\cr
&\qquad = \frac\e2\intt( \rho^\e|u^\e|^2) |_{t=0} \psi(0)\,dx - \int_0^T\intt (\nabla W \star \rho^\e)\rho^\e u^\e \psi(t)\,dx dt
\end{aligned}
\end{align}
for any nonnegative function $\psi(t) \in \mc[0,T]\cap \W^{1,\infty}(0,T)$ with $\psi(T) = 0$.
\end{itemize}
Here $L^1_+(\om)$ represents the set of nonnegative $L^1(\om)$ functions.
\end{definition}

\begin{definition}\label{def_weak2}For a given $T \in (0,\infty)$, we say that $(\rho^\e,\rho^\e u^\e)$ is a weak solution to the system \eqref{main_eq} if the following conditions are satisfied:
\begin{itemize}
\item[(i)] $\rho^\e \in \mc([0,T);L^1_+(\om))$, $\nabla W \star \rho^\e \in \mc([0,T);L^2(\om))$ and $\rho^\e |u^\e|^2 \in \mc([0,T);L^1(\om))$,
\item[(ii)] $(\rho^\e, \rho^\e u^\e)$ satisfies the system \eqref{main_eq} in the sense of distributions, and
\item[(iii)] $(\rho^\e, \rho^\e u^\e)$ satisfies the following weak formulation of energy equality:
\begin{align}\label{weak_form1}
\begin{aligned}
&-\frac\e2\int_0^T\intt \rho^\e|u^\e|^2 \psi'(t)\,dxdt + \gamma\int_0^T\intt \rho^\e|u^\e|^2 \psi(t)\,dxdt\cr
&\quad + \frac{1}{2}\int_0^T \inttt \phi(x-y)|u^\e(x) - u^\e(y)|^2 \rho^\e(x) \rho^\e(y)\psi(t)\,dxdydt\cr
&\qquad = \frac\e2\intt( \rho^\e|u^\e|^2) |_{t=0} \psi(0)\,dx - \int_0^T\intt (\nabla W \star \rho^\e)\rho^\e u^\e \psi(t)\,dx dt
\end{aligned}
\end{align}
for any nonnegative function $\psi(t) \in \mc[0,T]\cap \W^{1,\infty}(0,T)$ with $\psi(T) = 0$.
\end{itemize}
\end{definition}

\begin{definition}\label{def_strong}For given $T\in(0,\infty)$ and $p\in[1,2]$, we say that $(\rho,u)$ is a strong solution to the equation \eqref{main_eq2} if the following conditions are satisfied
\begin{itemize}
\item[(i)] $\rho \in \mc([0,T);(L^1_+ \cap \mathcal{P}_p)(\om))$,
\item[(ii)] $u \in L^\infty(0,T;\W^{1,\infty}(\om))$ and $\pa_t u \in L^\infty(\om \times (0,T))$, and
\item[(iii)] $\rho$ satisfies the system \eqref{main_eq2} in the sense of distributions.
\end{itemize}
\end{definition}

\begin{remark} If there is no velocity alignment force, i.e., $\phi \equiv 0$, then the condition (ii) in the Definition \ref{def_strong} can be replaced by
\begin{itemize}
\item[(ii)$^\prime$] $\pa_t u, \nabla u \in L^\infty(\om \times (0,T))$.
\end{itemize}
That is, the boundedness condition on $u$ can be removed. 
\end{remark}

\begin{remark} In the regularity interaction case, i.e., $\nabla W \in \W^{1,\infty}(\om)$, the condition (ii) in the Definition \ref{def_strong} can be removed. On the other hand, in the Coulomb interaction case, the condition (ii) can be replaced by
\begin{itemize}
\item[(ii)$^\prime$] $\rho \in L^\infty(0,T;\W^{1,p}(\om))$ and $\nabla \rho \in L^\infty(0,T;L^2(\om))$
for $d \geq 2$ and $\rho \in L^\infty(0,T;\W^{1,1}(\om))$ for $d=1$.
\end{itemize}
See Remark \ref{comm_u} for details.
\end{remark}

As mentioned in Introduction, the existence of weak solutions for the system \eqref{main_eq} with pressure is established in \cite{CFGS17} based on the methods of convex integration \cite{DS10}. This strategy also works for the pressureless case under more regular assumptions on the interaction potential $W$ and communication weight $\psi$, for instance $W \in \mc^2(\om)$ and $\psi \in \mc^1(\om)$, see \cite[Remark 2.1]{CFGS17} and \cite[Section 6]{CFGS17}, when $d=2$ or $3$. Moreover, when the interaction potential $W$ and the communication weight $\phi$ are bounded and Lipschitz continuous, we can also use a similar argument as in \cite[Theorem 2.4]{FST16} to obtain the global-in-time existence and uniqueness of solutions $(\rho,\rho u)$ to the equation \eqref{main_eq2}--\eqref{main_eq22} in the sense of Definition \ref{def_strong}. In fact, in this case, we only need to have the weak solution $\rho$ to the equations \eqref{main_eq2}--\eqref{main_eq22}. Concerning the condition Definition \ref{def_strong} (ii), see Remark \ref{comm_u} below. 

We next state our first main result showing the convergence of weak solutions $(\rho^\e, \rho^\e u^\e)$ of \eqref{main_eq} to a strong solution $(\rho,\rho u)$ of the equations \eqref{main_eq2}--\eqref{main_eq22} as $\e \to 0$ when the interaction potential $W$ is sufficiently regular.
\begin{theorem}\label{main_thm} Let $T>0$, $p \in [1,2]$, and $d\geq 1$. Let $(\rho^\e,\rho^\e u^\e)$ be a weak solution to the system \eqref{main_eq} in the sense of Definition \ref{def_weak} and $(\rho,u)$ be a strong solution to the equation \eqref{main_eq2} in the sense of Definition \ref{def_strong}. Suppose that the interaction potential $W$ and the communication weight $\phi$ are bounded and Lipschitz continuous. Moreover we assume that 
\[
\|\rho^\e_0\|_{L^1} =1 \quad \mbox{and} \quad \e \intt \rho^\e_0 |u^\e_0|^2\,dx \leq C \qquad \forall \, \e > 0
\]
for some $C>0$ independent of $\e>0$, and the strength of damping $\gamma > 0$ is large enough. Then we have
$$\begin{aligned}
&\int_0^T\intt \rho^\e(x,t) |(u^\e - u)(x,t)|^2\,dxdt + \sup_{0 \leq t \leq T} d_p^2(\rho^\e(\cdot,t) , \rho(\cdot,t))\cr
&\quad \leq C\e\intt \rho^\e_0(x) |(u^\e_0 - u_0)(x)|^2\,dx + Cd_p^2(\rho^\e_0(\cdot) , \rho_0(\cdot)) + C\e^2
\end{aligned}$$
and
$$\begin{aligned}
&\sup_{0 \leq t \leq T}\lt(\intt \rho^\e(x,t) |(u^\e - u)(x,t)|^2\,dx + d_p^2(\rho^\e(\cdot,t) , \rho(\cdot,t))\rt)\cr
&\quad \leq C\intt \rho^\e_0(x) |(u^\e_0 - u_0)(x)|^2\,dx + \frac{C}{\e}d_p^2(\rho^\e_0(\cdot) , \rho_0(\cdot)) + C\e,
\end{aligned}$$
where $C>0$ is independent of $\e > 0$. In particular, if we assume that
\[
\intt \rho^\e_0(x) |(u^\e_0 - u_0)(x)|^2\,dx + d_p(\rho^\e_0(\cdot) , \rho_0(\cdot)) = \mathcal{O}(\e),
\]
then we have
\[
\int_0^T\intt \rho^\e(x,t) |(u^\e - u)(x,t)|^2\,dxdt + \sup_{0 \leq t \leq T} d_p^2(\rho^\e(\cdot,t) , \rho(\cdot,t)) \leq C\e^2
\]
and
\bq\label{est_mr}
\sup_{0 \leq t \leq T}\lt(\intt \rho^\e(x,t) |(u^\e - u)(x,t)|^2\,dx + d_p^2(\rho^\e(\cdot,t) , \rho(\cdot,t))\rt) \leq C\e,
\eq
where $C>0$ is independent of $\e > 0$.
\end{theorem}

Our second result is on the Coulomb interaction case, i.e., the interaction potential $W$ satisfies $-\Delta W = \delta_0$.

\begin{theorem}\label{main_thm2} Let $T>0$ and $d\geq 1$. Let $(\rho^\e,\rho^\e u^\e)$ be a weak solution to the system \eqref{main_eq} in the sense of Definition \ref{def_weak2} and $(\rho,u)$ be a strong solution to the equations \eqref{main_eq2}--\eqref{main_eq22} in the sense of Definition \ref{def_strong}. Suppose that the communication weight $\phi$ are bounded and Lipschitz continuous and the interaction potential $W$ satisfies $-\Delta W = \delta_0$. 
Moreover we assume that 
\[
\|\rho^\e_0\|_{L^1} =1, \quad \e \intt \rho^\e_0 |u^\e_0|^2\,dx \leq C,  \quad \mbox{and} \quad \intt (W \star \rho^\e_0) \rho^\e_0\,dx \leq C \qquad \forall \, \e > 0
\]
for some $C>0$ independent of $\e>0$, and the strength of damping $\gamma > 0$ is large enough. In case $d=2$, we further suppose
\[
\intt \rho^\e_0 |x|^2\,dx \leq C \qquad \forall \, \e > 0
\]
for some $C>0$ independent of $\e>0$. Then we have
$$\begin{aligned}
&\int_0^T \intt \rho^\e(x,t) |(u^\e - u)(x,t)|^2\,dx dt +\sup_{0 \leq t \leq T}\intt |\nabla W \star (\rho - \rho^\e)(x,t)|^2\,dx \cr
&\quad \leq C\e\intt \rho^\e_0(x) |(u^\e_0 - u_0)(x)|^2\,dx +  C\intt |\nabla W \star (\rho_0 - \rho^\e_0)(x)|^2\,dx + C\e^2
\end{aligned}$$
and
$$\begin{aligned}
&\sup_{0 \leq t \leq T}\lt(\intt \rho^\e(x,t) |(u^\e - u)(x,t)|^2\,dx + \frac{1}{\e}\intt |\nabla W \star (\rho - \rho^\e)(x,t)|^2\,dx\rt) \cr
&\quad \leq C\intt \rho^\e_0(x) |(u^\e_0 - u_0)(x)|^2\,dx +  \frac{C}{\e}\intt |\nabla W \star (\rho_0 - \rho^\e_0)(x)|^2\,dx + C\e,
\end{aligned}$$
where $C>0$ is independent of $\e>0$. In particular, if we assume that
\[
\intt \rho^\e_0(x) |(u^\e_0 - u_0)(x)|^2\,dx = \mathcal{O}(\e)
\]
and
\[
\intt |\nabla W \star (\rho_0 - \rho^\e_0)(x)|^2\,dx = \mathcal{O}(\e^2),
\]
then we have
\[
\int_0^T \intt \rho^\e(x,t) |(u^\e - u)(x,t)|^2\,dx dt +\sup_{0 \leq t \leq T}\intt |\nabla W \star (\rho - \rho^\e)(x,t)|^2\,dx  \leq C\e^2
\]
and
\[
\sup_{0 \leq t \leq T}\lt(\intt \rho^\e(x,t) |(u^\e - u)(x,t)|^2\,dx + \frac{1}{\e}\intt |\nabla W \star (\rho - \rho^\e)(x,t)|^2\,dx\rt)  \leq C\e.
\]
\end{theorem}

\begin{remark}The condition $\|\rho^\e_0\|_{L^1} =1$ for all $\e>0$ in the above theorems can be replaced by $\|\rho^\e_0\|_{L^1} \leq M$ with $M>0$ independent of $\e>0$.
\end{remark}

\begin{remark} Inspired by \cite{CPW20,LT13,LT17}, even for the Coulomb interaction case, we provide quantitative error estimates in Theorem \ref{main_thm2} by using the weak formulation of energy equality \eqref{weak_form1}. However, we can also directly estimate by using the classical solutions to the system \eqref{main_eq}. We refer to \cite{CCJpre, CJpre} for the strong solvability of pressureless Euler-type equations.
\end{remark}

\begin{remark}The upper bound estimate of relative entropy in Theorem \ref{main_thm2} also provides the bound estimate in $p$-Wasserstein distance with $p \in [1,2]$, see Section \ref{sec_wasser} for details, or simply see \eqref{mkr_re}. Moreover, the bound estimate of $\|\nabla W \star (\rho - \rho^\e)(x,t)\|_{L^2}$ also gives the bound of $\|\rho - \rho^\e\|_{H^{-1}}$. Indeed, for any $\psi\in \dot{H}^1(\T^d)$ with $\|\psi\|_{\dot{H}^1} \leq1$, by using the integration by parts and H\"older's inequality, we find
$$\begin{aligned}
\lt|\intt \psi(x) (\rho - \rho^\e)(x)\,dx\rt| &= \lt|\intt \psi(x) (\Delta W \star (\rho - \rho^\e))(x)\,dx\rt| \cr
&= \lt|\intt \nabla \psi(x) \cdot(\nabla W \star (\rho - \rho^\e))(x)\,dx\rt|\cr
&\leq \intt |\nabla \psi(x)| |(\nabla W \star (\rho - \rho^\e))(x)|\,dx\cr
&\leq \|\nabla W \star (\rho - \rho^\e)(x,t)\|_{L^2},
\end{aligned}$$
that is, 
\[
\|\rho - \rho^\e\|_{H^{-1}} \leq \|\nabla W \star (\rho - \rho^\e)(x,t)\|_{L^2}.
\]
\end{remark}

\begin{remark} The estimate \eqref{est_mr} in Theorem \ref{main_thm} gives
\[
\sup_{0 \leq t \leq T}d_{BL}((\rho^\e u^\e)(\cdot,t),(\rho u)(\cdot,t)) \to 0
\]
as $\e \to 0$. Here $d_{BL}$ denotes the bounded Lipschitz distance. Indeed, for any $\varphi \in (L^\infty \cap Lip)(\om)$ we estimate
$$\begin{aligned}
\lt|\intt \lt((\rho^\e u^\e)(x) - (\rho u)(x)\rt) \varphi(x)\,dx\rt|&= \lt|\intt (\rho^\e(x) (u^\e - u)(x) \varphi(x)\,dx + \intt(\rho^\e - \rho)(x)u(x) \varphi(x)\,dx\rt|\cr
& \leq \|\varphi\|_{L^\infty} \lt(\intt \rho^\e(x)\,dx\rt)^{1/2}\lt(\intt \rho^\e(x)|u^\e(x)-u(x)|^2\,dx\rt)^{1/2} \cr
&\quad + \|u \varphi\|_{L^\infty \cap Lip} \,d_1(\rho^\e(\cdot),\rho(\cdot))\cr
&\leq C\lt(\intt \rho^\e(x)|u^\e(x)-u(x)|^2\,dx\rt)^{1/2} + Cd_1(\rho^\e(\cdot),\rho(\cdot))\cr
&\leq C\sqrt\e,
\end{aligned}$$
where $C>0$ is independent of $\e>0$. This together with Theorem \ref{main_thm} provides the limit from \eqref{main_eq} to \eqref{main_eq2}.
\end{remark}

\begin{remark} In the periodic domain case, $\om = \T^d$, the solution $\varphi$ to the following Poisson equation 
\bq\label{poi_eq}
-\Delta \varphi = \rho
\eq
cannot be expressed as $\nabla W \star \rho$ for some potential function $W$ with $-\Delta W = \delta_0$. Thus, in the case $\om = \T^d$ the term $-\nabla W \star \rho$ can be replaced with $\nabla \varphi$, where $\varphi$ solves \eqref{poi_eq}. According to this change, we have the results in Theorem \ref{main_thm2} with the substitution $\|\nabla (\varphi - \varphi^\e)\|_{L^2}^2$ for $\|\nabla W \star (\rho - \rho^\e)(x,t)\|_{L^2}^2$, where $\varphi^\e$ is the solution to \eqref{poi_eq} with $\rho^\e$. However, in order to simply the presentation of our work, not to write the pressureless Euler-Poisson equations in the periodic domain, we only consider the form of system \eqref{main_eq}.
\end{remark}

\begin{remark}\label{comm_u} Let us comment on the regularity assumptions on $u$ appeared in Definition \ref{def_strong} (ii). In fact, we show that $\|u\|_{L^\infty(0,T;\W^{1,\infty})}$ and $\|\pa_t u\|_{L^\infty}$ can be bounded from above by some constant which depends only on $\|\nabla W \star \rho\|_{L^\infty(0,T;\W^{1,\infty})}$, $\|\phi\|_{\W^{1,\infty}}$, $\|\rho\|_{L^\infty(0,T;L^1)}$, and $\gamma > 0$ when the strength of damping $\gamma > 0$ is sufficiently large. Since those estimates are rather lengthy and technical, for the smooth flow of reading we leave them in Appendix \ref{app_a}. Note that if the interaction potential $W$ satisfies $\nabla W \in \W^{1,\infty}(\om)$, then it readily follows $\|\nabla W \star \rho\|_{L^\infty(0,T;\W^{1,\infty})} \leq \|\nabla W\|_{\W^{1,\infty}}\|\rho\|_{L^\infty(0,T;L^1)}$. On the other hand, for the Coulomb interaction potential, if $d\geq 2$, we estimate
$$\begin{aligned}
\lt|\intt \nabla W (x-y) \rho(y)\,dy \rt| &\leq C\lt(\int_{|x-y| \geq 1} + \int_{|x-y| \leq 1}\rt) \frac{1}{|x-y|^{d-1}} \rho(y)\,dy\cr
&\leq C\lt(\|\rho\|_{L^1}+ \|\rho\|_{L^p}\rt)
\end{aligned}$$
and
$$\begin{aligned}
\lt|\intt \nabla W (x-y) \nabla \rho(y)\,dy \rt| &\leq C\lt(\int_{|x-y| \geq 1} + \int_{|x-y| \leq 1}\rt) \frac{1}{|x-y|^{d-1}} |\nabla \rho(y)|\,dy\cr
&\leq C\int_{|x-y| \geq 1} \frac{1}{|x-y|^{\frac{d-1}{2}}} |\nabla \rho(y)|\,dy + \int_{|x-y| \leq 1}\frac{1}{|x-y|^{d-1}} |\nabla \rho(y)|\,dy\cr
&\leq C\lt(\|\nabla \rho\|_{L^2}+ \|\nabla \rho\|_{L^p}\rt)
\end{aligned}$$
for some $p > d$. This yields
\[
\|\nabla W \star \rho\|_{L^\infty(0,T;\W^{1,\infty})}  \leq C\lt(\|\rho\|_{L^\infty(0,T;L^1)} + \|\nabla \rho\|_{L^\infty(0,T;L^2)} + \|\rho\|_{L^\infty(0,T;\W^{1,p})}\rt)
\]
when $W$ satisfies $-\Delta W = \delta_0$ with $d\geq 2$. In the case $d = 1$, $\|\nabla W\|_{L^\infty} \leq 1$ and thus 
\[
\|\nabla W \star \rho\|_{L^\infty(0,T;\W^{1,\infty})} \leq \|\rho\|_{L^\infty(0,T;\W^{1,1})}.
\]
\end{remark}

%
%
%
%
%
%
%
\section{$p$-Wasserstein distance and relative entropy functional}\label{sec_wasser}

In this section, we provide some relation between the $p$-Wasserstein distance and the relative entropy-type functional. In particular, if $p \in [1,2]$, $p$-Wasserstein distance can be bounded from above by our relative entropy functional. 

We state our main result of this section.

\begin{proposition}\label{prop_gd} Let $T>0$, $p \in[1,\infty]$, and $\bar \rho : [0,T] \to \mathcal{P}(\om)$ be a narrowly continuous solution of 
\[
\pa_t \bar\rho + \nabla \cdot (\bar\rho \bar u) = 0,
\] 
that is, $\bar\rho$ is continuous in the duality with continuous bounded functions, for a Borel vector field $\bar u$ satisfying
\bq\label{est_p1}
\int_0^T\intt |\bar u(x,t)|^p\bar\rho(x,t)\,dx dt < \infty.
\eq
Let $\rho \in \mc([0,T];\mathcal{P}_p(\om))$ be a solution of the following continuity equation:
\bq\label{conti_brho}
\pa_t \rho + \nabla \cdot (\rho u) = 0
\eq
with the velocity fields $u \in L^\infty(0,T; \dot{\W}^{1,\infty}(\om))$. Then there exists a positive constant $C$ depending only on $T$ such that for all $t \in [0,T]$
$$\begin{aligned}
&d_p(\bar \rho(\cdot,t), \rho(\cdot,t))   \leq Ce^{C\|\nabla u\|_{L^\infty}}\lt( d_p(\bar\rho(0), \rho(0)) +\lt(\int_0^t \intt |\bar u(x,s) - u(x,s)|^p \bar \rho(x,s)\,dxds\rt)^{1/p}\rt)
\end{aligned}$$
for $p \in [1,\infty)$, and
$$\begin{aligned}
&d_\infty(\bar \rho(\cdot,t), \rho(\cdot,t))  \leq Ce^{C\|\nabla u\|_{L^\infty}}\lt( d_\infty(\bar\rho(0), \rho(0)) + \sup_{s \in [0,T]}\esssup_{x \in supp(\bar \rho(s))}  |\bar u(x,s) - u(x,s)| \rt).
\end{aligned}$$
In particular, if $p \in [1,2]$, we have
$$\begin{aligned}
&d_p(\bar \rho(\cdot,t), \rho(\cdot,t))  \leq Ce^{C\|\nabla u\|_{L^\infty}}\lt( d_p(\bar\rho(0), \rho(0)) +\lt(\int_0^t \intt |\bar u(x,s) - u(x,s)|^2\bar \rho(x,s)\,dxds\rt)^{1/2}\rt),
\end{aligned}$$
where $C>0$ depends only on $T$.
\end{proposition}
\begin{proof} Since the proof is rather lengthy, we divide it into three steps for the sake of the reader. 

\begin{itemize}
\item In {\bf Step A}, we define the forward characteristics $X(t) := X(t;0,x)$ associated to the continuity equation \eqref{conti_brho}, that is, $X$ solves the following differential equations:
\bq\label{eq_char2}
\pa_t X(t) =  u(X(t),t) 
\eq
with the initial data $X(0) = x \in \om$. By using the above characteristic $X$, we introduce a density $\hat \rho$ which is determined by the push-forward of $\bar \rho(0)$ through the flow map $X$. Then we show 
\[
d_p(\rho(\cdot,t), \hat\rho(\cdot,t))  \leq  e^{\|\nabla u\|_{L^\infty} T}d_p(\rho(0), \bar \rho(0)).  
\]

\item In {\bf Step B}, we provide the quantitative bound for the error between $\bar \rho(t)$ and $\hat \rho(t)$ in the $p$-Wasserstein distance. More precisely, we show
\[
{} \hspace{1.5cm} d_p(\bar \rho(\cdot,t), \hat\rho(\cdot,t)) \leq Ct^{1-1/p}e^{C\|\nabla u\|_{L^\infty}}\lt(\int_0^t \intt |\bar u(x,s) - u(x,s)|^p \bar\rho(x,s)\,dxds\rt)^{1/p},
\]
where $C>0$ depends only on $T$.  \newline
 
\item In {\bf Step C}, we combine the estimates in the previous steps to conclude our desired results. \newline
\end{itemize}

{\bf Step A.-}  We first notice that the characteristic equations \eqref{eq_char2} are well-defined on the time interval $[0,T]$ since $u$ is bounded and Lipschitz continuous on $[0,T]$. To be more specific, there exists a unique solution $\rho$, which is determined as the push-forward of its initial density $\rho(0)$ through the flow maps $X$, i.e.,  $\rho(t) = X(t;0,\cdot) \# \rho(0)$. Here $\cdot \,\# \,\cdot $ stands for the push-forward of a probability measure by a measurable map, more precisely, $\nu = \mt \# \mu$ for probability measure $\mu$ and measurable map $\mt$ implies
\[
\intt \varphi(y) \,d\nu(y) = \intt \varphi(\mt(x)) \,d\mu(x)
\]
for all $\varphi \in \mc_b(\om)$. Moreover, by using the regularity of $u$, we can estimate the Lipschitz continuity of $X$ in $x$ as 
\begin{align*}
|X(t;0,x) - X(t;0,y)| &\leq |x-y| + \int_0^t |u(X(s;0,x)) - u(X(s;0,y))|\,ds\cr
&\leq |x-y| + \|\nabla u\|_{L^\infty}\int_0^t |X(s;0,x) - X(s;0,y)|\,ds.
\end{align*}
This together with applying Gr\"onwall's lemma gives
\bq\label{est_xlip}
|X(t;0,x) - X(t;0,y)| \leq  e^{\|\nabla u\|_{L^\infty}T}|x-y|.
\eq
This shows that the characteristic $X(t;0,x)$ is Lipschitz in $x$ with the Lipschitz constant $e^{\|\nabla u\|_{L^\infty}T}$. Let us now consider the density $\hat \rho$ which is given as the push-forward of $\bar \rho(0)$ through the flow map $X$, i.e., $\hat\rho(t) = X(t;0,x) \# \bar \rho(0)$. We then choose an optimal transport map for $d_p$ denoted by $\mt_0(x)$ between $\rho(0)$ and $\bar\rho(0)$ such that $\rho(0) = \mt_0 \# \bar\rho(0)$. Then since $\rho = X \# \rho(0)$ and $\hat \rho = X \# \bar \rho(0)$, we find
\[
d_p^p(\rho(\cdot,t), \hat\rho(\cdot,t)) \leq \intt |X(t;0,x) - X(t;0,\mt_0(x))|^p  \bar\rho(x,0)\,dx.
\]
Then this together with the Lipschitz estimate of $X$ appeared in \eqref{est_xlip} asserts
 \[
d_p^p(\rho(\cdot,t), \hat\rho(\cdot,t)) \leq  e^{p\|\nabla  u\|_{L^\infty} T}\intt |x - \mt_0(x)|^p  \rho(x,0)\,dx = e^{p\|\nabla u\|_{L^\infty} T}d_p^p(\rho(0), \bar \rho(0)),
 \]
 that is,
 \[
d_p(\rho(\cdot,t), \hat\rho(\cdot,t))  \leq  e^{\|\nabla u\|_{L^\infty} T}d_p(\rho(0), \bar \rho(0)).
 \]
 
{\bf Step B.-} It follows from \cite[Theorem 8.2.1]{AGS08}, see also \cite[Proposition 3.3]{FK19}, that
there exists a probability measure $\eta$ on $\Xi_T \times \om$ satisfying the following properties:
\begin{itemize}
\item[(i)] $\eta$ is concentrated on the set of pairs $(\xi,x)$ such that $\xi$ is an absolutely continuous curve satisfying
\bq\label{eq_gam}
\dot\xi(t) = \bar u(\xi(t),t)
\eq
for almost everywhere $t \in (0,T)$ with $\xi(0) = x \in \om$.
\item[(ii)] $\bar\rho$ satisfies
\bq\label{eq_gam2}
\intt \varphi(x)\bar\rho\,dx = \iint_{\Xi_T \times \T^d}\varphi(\xi(t))\,d\eta(\xi,x)
\eq
for all $\varphi \in \mc_b(\om)$, $t \in [0,T]$.
\end{itemize}
Then we use the disintegration theorem of measures (see \cite{AGS08} for instance) to write 
\[
d\eta(\xi,x) = \eta_x(d\xi) \otimes \bar\rho(x,0)\,dx,
\]
where $\{\eta_x\}_{x \in \om}$ is a family of probability measures on $\Xi_T$ concentrated on solutions of \eqref{eq_gam}. We then introduce a measure $\nu$ on $\Xi_T \times \Xi_T \times \om$ defined by
\[
d\nu(\xi, x, \sigma) = \eta_x(d\xi) \otimes \delta_{X(\cdot;0,x)}(d\sigma) \otimes \bar \rho(x,0)\,dx.
\]
We also introduce an evaluation map $E_t : \Xi_T  \times \Xi_T \times \om \to \om \times \om$ defined as $E_t(\xi, \sigma, x) = (\xi(t), \sigma(t))$. Then we readily show that measure $\pi_t:= (E_t)\# \nu$ on $\om \times \om$ has marginals $\bar\rho(x,t)\,dx$ and $\hat\rho(y,t)\,dy$ for $t \in [0,T]$, see \eqref{eq_gam2}. This implies
\begin{align}\label{est_rho2}
\begin{aligned}
d_p^p(\bar\rho(\cdot,t), \hat\rho(\cdot,t)) &\leq \inttt |x-y|^p\,d\pi_t(x,y)\cr
&=\iiint_{\Xi_T \times \Xi_T \times \om} |\sigma(t) - \xi(t) |^p \,d\nu(\xi, \sigma, x) \cr
&= \iint_{\Xi_T \times \om} |X(t;0,x) - \xi(t)|^p \,d\eta(\xi,x).
\end{aligned}
\end{align}
In order to estimate the right hand side of \eqref{est_rho2}, we use \eqref{eq_char2} and \eqref{eq_gam} to have
\begin{align*}
\lt|X(t;0,x) -\xi(t)\rt| &= \lt|\int_0^t u(X(s;0,x)) - \bar u(\xi(s),s)\,ds\rt|\cr
&\quad \leq \int_0^t \lt|u(X(s;0,x)) - u(\xi(s),s)\rt|ds + \int_0^t \lt|u(\xi(s),s) - \bar u(\xi(s),s)\rt|ds\cr
&\quad \leq \|\nabla u\|_{L^\infty}\int_0^t  \lt|X(s;0,x) - \xi(s)\rt|ds + \int_0^t \lt| u(\xi(s),s) - \bar u(\xi(s),s)\rt|ds,
\end{align*}
and subsequently, this yields
\[
\lt|X(t;0,x) -\xi(t)\rt| \leq Ce^{C\|\nabla u\|_{L^\infty}}\int_0^t \lt|u(\xi(s),s) - \bar u(\xi(s),s)\rt|ds,
\]
where $C>0$ is independent of $\e>0$. Combining this with \eqref{est_rho2}, we have
$$\begin{aligned}
d_p^p(\bar \rho(\cdot,t), \hat\rho(\cdot,t)) &\leq Ce^{Cp\|\nabla u\|_{L^\infty}}\iint_{\Xi_T \times \om} \lt|\int_0^t \lt|\bar u(\xi(s),s) - u(\xi(s),s)\rt|ds\rt|^p d\eta(\xi,x)\cr
&\leq Ct^{p-1}e^{Cp\|\nabla  u\|_{L^\infty}}\int_0^t\iint_{\Xi_T \times \om} \lt|\bar u(\xi(s),s) - u(\xi(s),s)\rt|^p d\eta(\xi,x)\,ds\cr
&\leq Ct^{p-1}e^{Cp\|\nabla u\|_{L^\infty}}\int_0^t \intt |\bar u(x,s) - u(x,s)|^p \bar\rho(x,s)\,dxds,
\end{aligned}$$
where $C>0$ is independent of $\e > 0$ and $p$, and we used the relation \eqref{eq_gam2}. This asserts
\[
d_p(\bar \rho(\cdot,t), \hat\rho(\cdot,t)) \leq Ct^{1-1/p}e^{C\|\nabla u\|_{L^\infty}}\lt(\int_0^t \intt |\bar u(x,s) - u(x,s)|^p \bar\rho(x,s)\,dxds\rt)^{1/p},
\]
where $C>0$ depends only on $T$.  \newline

{\bf Step C.-} Combining the estimates in {\bf Step A} \& {\bf Step B} yields
$$\begin{aligned}
d_p(\bar \rho(\cdot,t), \rho(\cdot,t)) 
&\leq d_p(\bar \rho(\cdot,t), \hat\rho(\cdot,t)) + d_p(\rho(\cdot,t), \hat\rho(\cdot,t))\cr
&\leq Ce^{C\|\nabla u\|_{L^\infty}}\lt( d_p(\bar\rho(0), \rho(0)) + \lt(\int_0^t \intt |\bar u(x,s) - u(x,s)|^p \bar\rho(x,s)\,dxds\rt)^{1/p}\rt),
\end{aligned}$$
where $C>0$ depends only on $T$. This provides the first assertion. Since the constant $C>0$ which appears in the above does not depend on $p$, after taking the supremum over the support of $\rho$ and the time interval $[0,T]$, we can pass to the limit $p \to \infty$ to derive the second assertion. Finally, if $p \in [1,2]$, then by using H\"older inequality the integral term on the right hand side of the above inequality can be estimated as
$$\begin{aligned}
&\int_0^t \intt |\bar u(x,s) - u(x,s)|^p \rho(x,s)\,dxds  \leq t^{1-p/2} \lt( \int_0^t \intt |\bar u(x,s) - u(x,s)|^2 \rho(x,s)\,dxds\rt)^{p/2}.
\end{aligned}$$
Hence we have
$$\begin{aligned}
&d_p(\bar \rho(t), \rho(t))  \leq Ce^{C\|\nabla \bar u\|_{L^\infty}}\lt( d_p(\bar\rho(0), \rho(0)) +\lt(\int_0^t \intt |\bar u(x,s) - u(x,s)|^2 \rho(x,s)\,dxds\rt)^{1/2}\rt).
\end{aligned}$$
This completes the proof.
\end{proof}

%
%
%
%
%
%
%

\section{Proof of Theorem \ref{main_thm}: Regular interaction case}\label{sec_pf}

In this section, we provide the details of proof for Theorem \ref{main_thm}. For this, we first estimate the relative entropy by using the weak formulation.

\begin{proposition} Let $(\rho^\e,\rho^\e u^\e)$ be a weak solution to the system \eqref{main_eq} in the sense of Definitions \ref{def_weak} and \ref{def_weak2} when the interaction potential $W$ satisfies $\nabla W \in \W^{1,\infty}$ and $-\Delta W = \delta_0$, respectively, and $(\rho,u)$ be a strong solution to the equation \eqref{main_eq2}--\eqref{main_eq22}  in the sense of Definition \ref{def_strong}. Then we have
\begin{align}\label{est_re0}
\begin{aligned}
&\frac12\intt \rho^\e |u^\e - u|^2\,dx\bigg|_{\tau = 0}^{\tau = t} + \frac\gamma\e \int_0^t \intt \rho^\e |u^\e - u|^2\,dx d\tau\cr
&\quad = -\int_0^t  \intt \rho^\e \nabla u : (u^\e - u) \otimes (u^\e - u)\,dxd\tau  - \frac{1}{\e}\int_0^t \intt \rho^\e(u^\e - u) \cdot \nabla W \star (\rho^\e - \rho)\,dxd\tau\cr
&\qquad -\int_0^t \intt \rho^\e(u^\e - u) \cdot e\,dxd\tau  + \frac1\e\int_0^t \inttt \phi(x-y) \rho^\e(x) (u^\e(x) - u(x))\cr
&\hspace{6.5cm}  \cdot \lt( (u^\e(y) - u^\e(x))\rho^\e(y) - (u(y) - u(x))\rho(y)  \rt)dxdyd\tau.
\end{aligned}
\end{align}
\end{proposition}

\begin{proof} Although this proof is very similar to \cite{CPW20,LT13,LT17}, for the completeness of our work, we provide the details. Let us take the following test function
\begin{displaymath}
\psi(\tau) = \left\{ \begin{array}{ll}
1 & \textrm{for $0 \leq \tau < t$}\\[2mm]
\displaystyle \frac{t-\tau}{k} + 1 & \textrm{for $t \leq \tau < t + k$}\\[2mm]
 0 & \textrm{for $\tau \geq t+ k$}
  \end{array} \right.
\end{displaymath}
in \eqref{weak_form1} to obtain
\begin{align*}
\begin{aligned}
&\frac{\e}{2k}\int_t^{t+k} \intt \rho^\e |u^\e|^2\,dxd\tau + \gamma\int_0^t \intt \rho^\e |u^\e|^2\,dxd\tau \cr
&\quad + \frac{1}{2} \int_0^t  \inttt \phi(x-y)|u^\e(x) - u^\e(y)|^2 \rho^\e(x) \rho^\e(y)\,dxdyd\tau + \int_t^{t+k} \intt \lt(\frac{t-\tau}{k} + 1 \rt)\rho^\e |u^\e|^2\,dxd\tau\cr
&\quad + \frac{1}{2}\int_t^{t+k}  \inttt \lt(\frac{t-\tau}{k} + 1 \rt) \phi(x-y)|u^\e(x) - u^\e(y)|^2 \rho^\e(x) \rho^\e(y)\,dxdyd\tau\cr
&\qquad = \frac\e2\intt( \rho^\e|u^\e|^2) |_{\tau=0} \,dx - \int_0^t\intt (\nabla W \star \rho^\e)\rho^\e u^\e \,dx d\tau - \int_t^{t+k}\intt \lt(\frac{t-\tau}{k} + 1 \rt)  (\nabla W \star \rho^\e)\rho^\e u^\e \,dx d\tau.
\end{aligned}
\end{align*}
We then send $k \to 0+$ and divide the resulting equation by $\e$ to derive the kinetic energy estimate:
\begin{align}\label{est_f}
\begin{aligned}
& \frac12\intt \rho^\e|u^\e|^2  \,dx \bigg|_{\tau=0}^{\tau = t} + \frac\gamma\e\int_0^t \intt \rho^\e |u^\e|^2\,dxd\tau  \cr
&\quad + \frac{1}{2\e} \int_0^t  \inttt \phi(x-y)|u^\e(x) - u^\e(y)|^2 \rho^\e(x) \rho^\e(y)\,dxdyd\tau\cr
 &\qquad = -\frac{1}{\e} \int_0^t\intt (\nabla W \star \rho^\e) \cdot \rho^\e u^\e \,dx d\tau.
 \end{aligned}
\end{align}
In a similar fashion, we can also estimate the kinetic energy for the limiting system \eqref{main_eq2} as
\begin{align*}
\begin{aligned}
&\frac12\intt \rho|u|^2 \,dx  \bigg|_{\tau=0}^{\tau = t} + \gamma\int_0^t \intt \rho |u|^2\,dxd\tau  + \frac{1}{2} \int_0^t  \inttt \phi(x-y)|u(x) - u(y)|^2 \rho(x) \rho(y)\,dxdyd\tau\cr
 &\qquad =  -\int_0^t\intt (\nabla W \star \rho)\rho u \,dx d\tau + \int_0^t\intt \rho e \cdot u\,dxd\tau,
 \end{aligned}
\end{align*}
where $e =  \pa_t u + (u \cdot \nabla) u$. On the other hand, it follows from \eqref{main_eq} and \eqref{main_eq2}--\eqref{main_eq22} that
\bq\label{dconti}
\pa_t (\rho^\e - \rho) + \nabla \cdot (\rho^\e u^\e - \rho u)  = 0
\eq
and
\begin{align}\label{dmom}
\begin{aligned}
\pa_t (\rho^\e u^\e - \rho u) + \nabla \cdot (\rho^\e u^\e \otimes u^\e - \rho u\otimes u)& = -\frac{\gamma}{\e}(\rho^\e u^\e - \rho u) - \frac1\e \lt( (\nabla W \star \rho^\e) \rho^\e - (\nabla W \star \rho)\rho\rt)\cr
&\quad + \frac1\e \rho^\e \intt \phi(x-y) (u^\e(y) - u^\e(x))\rho^\e(y)\,dy \cr
&\quad -  \frac1\e \rho \intt \phi(x-y) (u(y) - u(x))\rho(y)\,dy - \rho e.
 \end{aligned}
\end{align}
Then we apply a test function $\varphi \in \mc([0,T]; \W^{1,\infty}(\om))$ with $\varphi(\cdot,T)=0$ to \eqref{dconti} to get
\begin{align*}
\begin{aligned}
&-\int_0^T \intt (\rho^\e - \rho)\pa_t \varphi\,dxdt - \int_0^T \intt \nabla \varphi \cdot (\rho^\e u^\e - \rho u)\,dxdt  = \intt \varphi (\rho^\e - \rho) \,dx\bigg|_{t=0} = 0.
 \end{aligned}
\end{align*}
Similarly, we also consider a test function $\bar\varphi \in \mc([0,T]; \W^{1,\infty}(\om))$ with $\bar\varphi(\cdot,T) = 0$ for \eqref{dmom} to yield
\begin{align*}
\begin{aligned}
&-\int_0^T \intt \pa_t\bar\varphi \cdot (\rho^\e u^\e - \rho u)\,dxdt - \int_0^T \intt \nabla \bar\varphi : (\rho^\e u^\e \otimes u^\e - \rho u \otimes u)\,dxdt\cr
&\quad = \intt \bar\varphi \cdot (\rho^\e u^\e - \rho u)\,dx\bigg|_{t=0} - \frac{\gamma}{\e} \int_0^T \intt \bar\varphi \cdot (\rho^\e u^\e - \rho u)\,dxdt\cr
&\qquad -\frac{1}{\e}\int_0^T\intt \bar\varphi \cdot \lt( (\nabla W \star \rho^\e) \rho^\e - (\nabla W \star \rho)\rho\rt)dx dt\cr
&\qquad + \frac1\e \int_0^T \inttt  \rho^\e(x) \phi(x-y) \bar\varphi \cdot (u^\e(y) - u^\e(x))\rho^\e(y)\,dxdydt \cr
&\qquad -  \frac1\e \int_0^T \inttt  \rho(x) \phi(x-y) \bar\varphi \cdot (u(y) - u(x))\rho(y)\,dxdydt.
 \end{aligned}
\end{align*}
We then choose the following specific test functions:
\[
\varphi = -\psi(\tau) \frac{|u|^2}{2} \quad \mbox{and} \quad \bar\varphi = \psi(\tau) u.
\]
Then similarly as before, we find
\begin{align}\label{est_conti}
\begin{aligned}
&\intt \lt(-\frac{|u|^2}{2}(\rho^\e - \rho) \rt)dx \bigg|_{\tau=0}^{\tau = t} + \int_0^t \intt \pa_\tau \lt(\frac{|u|^2}{2} \rt) (\rho^\e - \rho)\,dxd\tau\cr
&\quad + \int_0^t \intt \nabla \lt( \frac{|u|^2}{2} \rt) \cdot (\rho^\e u^\e - \rho u)\,dxd\tau = 0
\end{aligned}
\end{align}
and
\begin{align}\label{est_mom}
\begin{aligned}
&\intt u \cdot (\rho^\e u^\e - \rho u)\,dx \bigg|_{\tau=0}^{\tau = t} - \int_0^t \intt \pa_\tau u \cdot (\rho^\e u^\e - \rho u)\,dxd\tau  - \int_0^t \intt \nabla u : (\rho^\e u^\e \otimes u^\e - \rho u\otimes u)\,dxd\tau\cr
&\qquad = -\frac{\gamma}{\e} \int_0^t \intt u \cdot (\rho^\e u^\e - \rho u)\,dxd\tau - \frac1\e\int_0^t \intt u \cdot \lt(( \nabla W \star \rho^\e) \rho^\e - (\nabla W \star \rho) \rho  \rt) dxd\tau\cr
&\quad \qquad + \frac1\e \int_0^T \inttt  \rho^\e(x) \phi(x-y)  u(x) \cdot (u^\e(y) - u^\e(x))\rho^\e(y)\,dxdydt \cr
&\quad \qquad -  \frac1\e \int_0^T \inttt  \rho(x) \phi(x-y)u(x) \cdot (u(y) - u(x))\rho(y)\,dxdydt.
\end{aligned}
\end{align}
In order to derive the relative entropy inequality, we notice that the velocity field $u$ of the equation \eqref{main_eq2} satisfies 
\begin{align*}
\begin{aligned}
\pa_\tau u + u \cdot \nabla u &= - \frac{\gamma}{\e} u - \frac1\e(\nabla W \star \rho) + \frac1\e\intt \phi(x-y) (u(y) - u(x))\rho(y)\,dy + e,
\end{aligned}
\end{align*}
where $e = \pa_\tau u + u \cdot \nabla u$. We then multiply the above by $\rho^\e (u^\e - u)$ to have
\begin{align*}
\begin{aligned}
&-(\rho^\e - \rho)\pa_\tau \lt(\frac{|u|^2}{2} \rt) + \pa_\tau u \cdot (\rho^\e u^\e - \rho u) - \nabla\lt(\frac{|u|^2}{2} \rt) (\rho^\e u^\e - \rho u) + \nabla u: (\rho^\e u^\e \otimes u^\e - \rho u\otimes u)\cr
&\qquad = \rho^\e\nabla u: (u^\e -u)\otimes (u^\e - u) - \frac1\e \rho^\e(\nabla W \star \rho)\cdot(u^\e - u)\cr
&\qquad \quad -\frac\gamma\e \rho^\e u \cdot (u^\e - u) + \frac1\e\rho^\e (u^\e - u) \cdot\intt \phi(x-y)  (u(y) - u(x))\rho(y)\,dy + \rho^\e  e \cdot (u^\e - u). 
\end{aligned}
\end{align*}
Then we integrate the above system over $\om \times [0,T]$, and finally combine it with \eqref{est_conti} and \eqref{est_mom} to conclude the desired result.
\end{proof}

\begin{remark}\label{rmk_free} If $\nabla W \in L^\infty(\om)$, then we estimate
\begin{align*}
\begin{aligned}
\lt|\int_0^t\intt (\nabla W \star \rho^\e)\cdot \rho^\e u^\e \,dx d\tau \rt| &\leq \|\nabla W\|_{L^\infty}\lt( \int_0^t\intt \rho^\e |u^\e| \,dx d\tau\rt)\cr
& \leq \|\nabla W\|_{L^\infty}T\lt( \int_0^t\intt \rho^\e |u^\e|^2 \,dx d\tau \rt)^{1/2}\cr
&\leq C + \frac\gamma2  \int_0^t\intt \rho^\e |u^\e|^2 \,dx d\tau
\end{aligned}
\end{align*}
due to $\|\rho^\e_0\|=1$, where $C>0$ depends only on $\|\nabla W\|_{L^\infty}$, $T$, and $\gamma>0$. This together with \eqref{est_f} yields
\[
\int_0^t \intt \rho^\e |u^\e|^2\,dxd\tau \leq C\e \intt \rho^\e_0 |u^\e_0|^2\,dx + C,
\]
where $C>0$ is independent of $\e$. 

In the Coulomb interaction case, we estimate 
\[
- \int_0^t\intt (\nabla W \star \rho^\e)\cdot \rho^\e u^\e \,dx d\tau  = \int_0^t\intt (W \star \rho^\e) (\nabla \cdot \rho^\e u^\e) \,dx d\tau = -\frac12 \intt (W\star\rho^\e) \rho^\e\,dx \bigg|_{\tau=0}^{\tau = t}.
\]
Thus we obtain from \eqref{est_f} that 
\bq\label{est_kin}
\gamma \int_0^t \intt \rho^\e |u^\e|^2\,dxd\tau \leq C\e \intt \rho^\e_0 |u^\e_0|^2\,dx + \intt (W\star\rho^\e_0) \rho^\e_0\,dx - \intt (W\star\rho^\e) \rho^\e\,dx.
\eq
Note that the third term on the right hand side of the above inequality is nonpositive when $d\geq 3$. If $d=1$, then $W = |x|$ and thus $\|\nabla W\|_{L^\infty} \leq 1$. Hence by the above estimates we find
\[
\int_0^t \intt \rho^\e |u^\e|^2\,dxd\tau \leq C\e \intt \rho^\e_0 |u^\e_0|^2\,dx + C.
\]
In case $d=2$, we have
\[
- \intt (W\star\rho^\e) \rho^\e\,dx \leq \frac1{2\pi} \iint_{|x-y|\geq 1} \log|x-y| \rho^\e(x) \rho^\e(y)\,dxdy \leq 4 \intt |x|^2 \rho^\e\,dx
\]
since $\log s \leq s^2$ for $ s\geq 1$. On the other hand, we can easily estimate
\[
\intt |x|^2 \rho^\e\,dx \leq C\intt |x|^2 \rho^\e_0\,dx + C\int_0^t \intt \rho^\e |u^\e|^2\,dxd\tau,
\]
and thus choosing $\gamma > 0$ large enough and combining this with \eqref{est_kin} give
\[
\gamma \int_0^t \intt \rho^\e |u^\e|^2\,dxd\tau \leq C\e \intt \rho^\e_0 |u^\e_0|^2\,dx + \intt (W\star\rho^\e_0) \rho^\e_0\,dx + C\intt |x|^2 \rho^\e_0\,dx.
\]
\end{remark}

We now provide the details of proof of Theorem \ref{main_thm}.

\begin{proof}[Proof of Theorem \ref{main_thm}] We first estimate the last term on the right hand side of \eqref{est_re0} as
\begin{align*}
\begin{aligned}
&\int_0^t \inttt \phi(x-y) \rho^\e(x) (u^\e(x) - u(x))  \cdot \lt( (u^\e(y) - u^\e(x))\rho^\e(y) - (u(y) - u(x))\rho(y)  \rt)dxdyd\tau\cr
&\quad = -\frac12\int_0^t\inttt \phi(x-y) \rho^\e(x) \rho^\e(y) |(u^\e(x) - u(x)) - (u^\e(y) - u(y)) |^2\,dxdyd\tau \cr
&\qquad + \int_0^t\inttt \phi(x-y) \rho^\e(x) (\rho^\e - \rho)(y) (u^\e(x) - u(x)) \cdot (u(y) - u(x))\,dxdyd\tau.
\end{aligned}
\end{align*}
This together with \eqref{est_re0} yields
\begin{align}\label{est_re0}
\begin{aligned}
&\frac12\intt \rho^\e |u^\e - u|^2\,dx + \frac\gamma\e \int_0^t \intt \rho^\e |u^\e - u|^2\,dx d\tau\cr
&\quad + \frac1{2\e}\int_0^t\inttt \phi(x-y) \rho^\e(x) \rho^\e(y) |(u^\e(x) - u(x)) - (u^\e(y) - u(y)) |^2\,dxdyd\tau\cr
&\quad = \frac12\intt \rho^\e_0 |u^\e_0 - u_0|^2\,dx -\int_0^t  \intt \rho^\e \nabla u : (u^\e - u) \otimes (u^\e - u)\,dxd\tau \cr
&\qquad - \frac{1}{\e}\int_0^t \intt \rho^\e(u^\e - u) \cdot \nabla W \star (\rho^\e - \rho)\,dxd\tau --\int_0^t \intt \rho^\e(u^\e - u) \cdot e\,dxd\tau \cr
&\qquad +  \frac1\e\int_0^t\inttt \phi(x-y) \rho^\e(x) (\rho^\e - \rho)(y) (u^\e(x) - u(x)) \cdot (u(y) - u(x))\,dxdyd\tau\cr
&\quad =: \sum_{i=1}^5 I_i.
\end{aligned}
\end{align}

$\diamond$ Estimate of $I_2$: We simply use the strong regularity assumptions on the solution $(\rho,u)$ to the limiting system \eqref{main_eq2}--\eqref{main_eq22} to get
\[
I_2 \leq \|\nabla u\|_{L^\infty}\int_0^t\intt \rho^\e |u^\e - u|^2\,dxd\tau.
\]

$\diamond$ Estimate of $I_3$: Note that
\[
\|\nabla W \star (\rho^\e - \rho)\|_{L^\infty} \leq \|\nabla W\|_{ Lip} \,d_1(\rho^\e,\rho) \leq \|\nabla W\|_{ Lip} \,\,d_p(\rho^\e,\rho)
\]
for $p \in [1,2]$. This gives
\begin{align*}
\begin{aligned}
I_3 &\leq \frac{\|\nabla W\|_{ Lip}}{\e}\int_0^t \lt(\intt \rho^\e |u^\e - u|^2\,dx \rt)^{1/2} \lt(\intt \rho^\e\,dx \rt)^{1/2} d_p (\rho^\e,\rho) \,d\tau\cr
&\leq \frac{\|\nabla W\|_{ Lip}}{\e} \lt( \int_0^t \intt \rho^\e |u^\e - u|^2\,dx d\tau\rt)^{1/2}\lt(\int_0^td_p^2 (\rho^\e,\rho) \,d\tau  \rt)^{1/2}\cr
&\leq \frac{C}{\e}  \int_0^t \intt \rho^\e |u^\e - u|^2\,dx d\tau + \frac{C}{\e} \int_0^td_p^2 (\rho^\e,\rho) \,d\tau.
\end{aligned}
\end{align*}

$\diamond$ Estimate of $I_4$: Recall $e = \pa_t u + (u \cdot \nabla) u$, and we estimate
\begin{align*}
\begin{aligned}
I_4 \leq \|\pa_t u\|_{L^\infty} \int_0^t \intt \rho^\e |u^\e - u|\,dx + \int_0^t \intt \rho^\e|u^\e - u| | u \cdot \nabla u|\,dxd\tau =: I_4^1 + I_4^2.
\end{aligned}
\end{align*}
Similarly to the estimate of $I_3$, we get
\begin{align*}
\begin{aligned}
I_4^1 &\leq \|\pa_t u\|_{L^\infty}\int_0^t \lt(\intt \rho^\e |u^\e - u|^2\,dx \rt)^{1/2} \lt(\intt \rho^\e\,dx \rt)^{1/2} \,d\tau\cr
&\leq  \|\pa_t u\|_{L^\infty} \sqrt{ T} \lt(\int_0^t \intt \rho^\e |u^\e - u|^2\,dx d\tau\rt)^{1/2}  \cr
&\leq\frac{C}{\e}\int_0^t\intt \rho^\e |u^\e - u|^2\,dxd\tau + C \e,
\end{aligned}
\end{align*}
where $C>0$ only depends on $\|\pa_t u\|_{L^\infty}, \|\rho_0^\e\|_{L^1}$, and $T$. 

For $I_4^2$, we find
\begin{align*}
\begin{aligned}
I_4^2 &= \int_0^t \intt \rho^\e|u^\e - u| | (u  - u^\e + u^\e)\cdot \nabla u|\,dxd\tau\cr
&\leq \|\nabla u\|_{L^\infty} \lt(\int_0^t \intt \rho^\e|u^\e - u|^2 \,dxd\tau + \lt(\int_0^t \intt \rho^\e |u^\e - u|^2\,dx d\tau\rt)^{1/2}\lt(\int_0^t \intt \rho^\e |u^\e|^2\,dx d\tau\rt)^{1/2} \rt).
\end{aligned}
\end{align*}
We then combine this estimate with the uniform bound estimate in Remark \ref{rmk_free}  to yield
\[
I_4^2 \leq \frac{C}{\e}\int_0^t \intt \rho^\e|u^\e - u|^2 \,dxd\tau + C\e
\]
for some $C>0$ independent of $\e>0$. This implies
\[
I_4 \leq \frac{C}{\e}\int_0^t \intt \rho^\e|u^\e - u|^2 \,dxd\tau + C\e.
\]

$\diamond$ Estimate of $I_5$: We divide $I_5$ into two terms:
\begin{align*}
\begin{aligned}
I_5 &= \frac{1}{\e}\int_0^t \intt \lt(\intt\phi(x-y)u(y) (\rho^\e - \rho)(y) \,dy\rt) \cdot (u^\e-u)(x)\rho^\e(x)\,dxd\tau\cr
&\quad + \frac{1}{\e} \int_0^t \intt \lt(\intt \phi(x-y)(\rho^\e - \rho)(y)\,dy\rt) (u^\e(x) - u(x)) \cdot u(x) \rho^\e(x)\,dxd\tau\cr
&=: I_5^1 + I_5^2.
\end{aligned}
\end{align*}
Here we use the regularity of $u$ and $\phi$, H\"older's inequality, the inequality $d_p \leq d_q$ for $1 \leq p \leq q \leq \infty$ to estimate
\begin{align*}
\begin{aligned}
I_5^1 &\leq \frac{\|\phi u\|_{Lip}}{\e}\int_0^t d_1(\rho^\e,\rho) \lt( \intt \rho^\e |u^\e - u|\,dx\rt)^{1/2}\,d\tau \cr
&\leq \frac{\|\phi u\|_{ Lip}}{\e}\int_0^t d_p(\rho^\e,\rho) \lt( \intt \rho^\e |u^\e - u|\,dx\rt)^{1/2}\,d\tau \cr
&\leq \frac{\|\phi u\|_{ Lip}}{\e}\lt(\int_0^t d_p^2(\rho^\e,\rho)\,d\tau\rt)^{1/2} \lt(\int_0^t\intt \rho^\e |u^\e - u|^2\,dxd\tau\rt)^{1/2}
\end{aligned}
\end{align*}
and
\begin{align*}
\begin{aligned}
I_5^2 &\leq \frac{\|\phi\|_{ Lip}}{\e} \int_0^t d_p(\rho^\e, \rho)\lt( \intt \rho^\e |u^\e - u||u|\,dx\rt)d\tau \cr
&\leq \frac{\|u\|_{L^\infty}\|\phi\|_{ Lip}}{\e} \lt(\int_0^t d_p^2(\rho^\e,\rho)\,d\tau\rt)^{1/2} \lt(\int_0^t\intt \rho^\e |u^\e - u|^2\,dxd\tau\rt)^{1/2}.
\end{aligned}
\end{align*}
This asserts
$$\begin{aligned}
I_5 &\leq \frac{C}{\e} \lt(\int_0^t d_p^2(\rho^\e,\rho)\,d\tau\rt)^{1/2} \lt(\int_0^t\intt \rho^\e |u^\e - u|^2\,dxd\tau\rt)^{1/2}\cr
&\leq \frac{C}{\e} \int_0^td_p^2 (\rho^\e,\rho) \,d\tau + \frac{C}{\e}  \int_0^t \intt \rho^\e |u^\e - u|^2\,dx d\tau,
\end{aligned}$$
where $C$ depends only on $\|\rho^\e_0\|_{L^1}, \|\phi\|_{\W^{1,\infty}}$, and $\|u\|_{\W^{1,\infty}}$. 

Now we combine all of the above estimates to have
\begin{align}\label{est_re0}
\begin{aligned}
&\frac12\intt \rho^\e |u^\e - u|^2\,dx + \frac{(\gamma-C)}{\e} \int_0^t \intt \rho^\e |u^\e - u|^2\,dx d\tau\cr
&\quad + \frac1{2\e}\int_0^t\inttt \phi(x-y) \rho^\e(x) \rho^\e(y) |(u^\e(x) - u(x)) - (u^\e(y) - u(y)) |^2\,dxdyd\tau\cr
&\qquad \leq \frac12\intt \rho^\e_0 |u^\e_0 - u_0|^2\,dx + \frac{C}{\e} \int_0^t d_p^2 (\rho^\e,\rho) \,d\tau + C\e
\end{aligned}
\end{align}
for $\e < 1$, where $C>0$ is independent of $\e>0$. 

Note that our solution $(\rho^\e, \rho^\e u^\e)$ has a bounded kinetic energy $\rho^\e |u^\e|^2 \in L^\infty(0,T;L^1(\om))$, that is, the integrability condition \eqref{est_p1} holds with $p=2$, thus by using the interpolation inequality, we can use Proposition \ref{prop_gd} with $(\bar \rho, \bar u) = (\rho^\e, u^\e)$ to estimate the $p$-Wasserstein distance as 
\bq\label{mkr_re}
d_p(\rho^\e(\cdot,t), \rho(\cdot,t)) \leq C\lt( d_p(\rho^\e_0, \rho_0) +\lt(\int_0^t \intt \rho^\e |u^\e - u|^2\,dx d\tau\rt)^{1/2}\rt),
\eq
where $C>0$ depends on $T$ and $\|u\|_{\W^{1,\infty}}$, but independent of $\e$. Putting \eqref{mkr_re} into \eqref{est_re0} yields
$$\begin{aligned}
&\frac12\intt \rho^\e |u^\e - u|^2\,dx + \frac{1}{\e} d_p^2(\rho^\e(\cdot,t), \rho(\cdot,t)) + \frac{(\gamma-C)}{\e} \int_0^t \intt \rho^\e |u^\e - u|^2\,dx d\tau\cr
&\quad + \frac1{2\e}\int_0^t\inttt \phi(x-y) \rho^\e(x) \rho^\e(y) |(u^\e(x) - u(x)) - (u^\e(y) - u(y)) |^2\,dxdyd\tau\cr
&\qquad \leq \frac12\intt \rho^\e_0 |u^\e_0 - u_0|^2\,dx + \frac{C}{\e} d_p^2 (\rho^\e_0,\rho_0) + C\e,
\end{aligned}$$
where $C>0$ is independent of $\e>0$. Then we again use the above estimate with \eqref{mkr_re} to have
$$\begin{aligned}
&\frac12\intt \rho^\e |u^\e - u|^2\,dx + \frac{(\gamma-C)}{\e} \int_0^t \intt \rho^\e |u^\e - u|^2\,dx d\tau\cr
&\quad + \frac1{2\e}\int_0^t\inttt \phi(x-y) \rho^\e(x) \rho^\e(y) |(u^\e(x) - u(x)) - (u^\e(y) - u(y)) |^2\,dxdyd\tau\cr
&\qquad \leq \frac12\intt \rho^\e_0 |u^\e_0 - u_0|^2\,dx + \frac{C}{\e} d_p^2 (\rho^\e_0,\rho_0) + C\e,
\end{aligned}$$
where $C>0$ is independent of $\e>0$. This completes the proof.
\end{proof}

%
%
%
%

\section{Proof of Theorem \ref{main_thm2}: Coulomb interaction case}\label{sec_poisson}
In this section, we are interested in the interaction potential function $W$ given as the fundamental solution of Laplace's equation, i.e., $-\Delta W = \delta_0$. In order to handle this case, we only need to estimate $I_3$ term in \eqref{est_re0} in a different way because of lack of regularity of $\nabla W$. We notice that the estimate of $I_4$ have already estimated even for the Coulomb interaction case in the proof of Theorem \ref{main_thm}. Motivated from \cite{CFGS17,LT17}, we provide the following lemma which allows us to change some part of the term $I_3$ to the time derivative of $L^2$ norm of $\nabla W \star (\rho - \rho^\e)$.

\begin{lemma}\label{lem_wd} Suppose that the interaction potential $W$ satisfies $-\Delta W =  \delta_0$. Then we have
\[
\frac12\frac{d}{dt}\intt |\nabla W \star (\rho - \rho^\e)|^2\,dx = \intt \nabla W\star(\rho - \rho^\e)  \cdot \lt((\rho u) - (\rho^\e u^\e)\rt) dx.
\]
\end{lemma}
\begin{proof}Using the equation for $\rho^\e - \rho$ in \eqref{dconti}, we find
\begin{align}\label{est_w1}
\begin{aligned}
\frac12\frac{d}{dt}\intt |\nabla W \star (\rho - \rho^\e)|^2\,dx &= \intt (\nabla W \star (\rho - \rho^\e))\cdot (\nabla W \star (\pa_t (\rho - \rho^\e)))\,dx\cr
&=- \intt (\Delta W \star (\rho - \rho^\e)) \lt(W\star (\pa_t (\rho - \rho^\e))\rt)dx\cr
&=  \intt  (\rho - \rho^\e) \lt(W\star (\pa_t (\rho -  \rho^\e))\rt)dx.\cr
\end{aligned}
\end{align}
We then use the symmetry of $W$ to get
\begin{align}\label{est_w2}
\begin{aligned}
&\intt (\rho - \rho^\e) \lt(W\star (\pa_t (\rho -  \rho^\e))\rt)dx\cr
&\quad =-\inttt (\rho - \rho^\e)(x) W(x-y) \lt( \nabla_y \cdot (\rho u)(y) - \nabla_y \cdot (\rho^\e u^\e)(y) \rt) dxdy\cr
&\quad = \inttt (\rho - \rho^\e)(x) \nabla_y \lt( W(x-y)\rt) \cdot \lt( (\rho u)(y) - (\rho^\e u^\e)(y) \rt) dxdy\cr
&\quad =-\inttt (\rho - \rho^\e)(x) \nabla_x W(x-y) \cdot \lt( (\rho u)(y) - (\rho^\e u^\e)(y) \rt) dxdy\cr
&\quad =\inttt (\rho - \rho^\e)(y) \nabla_x W(x-y) \cdot \lt( (\rho u)(x) - (\rho^\e u^\e)(x) \rt) dxdy\cr
&\quad = \intt \nabla W\star(\rho - \rho^\e)  \cdot \lt((\rho u) - (\rho^\e u^\e)\rt) dx.
\end{aligned}
\end{align}
We finally combine \eqref{est_w1} and \eqref{est_w2} to conclude the proof.
\end{proof}
\begin{proof}[Proof of Theorem \ref{main_thm2}] As mentioned above, we only need to estimate $I_3$ in \eqref{est_re0}. 
\begin{align}\label{I_3}
\begin{aligned}
I_3 &= \frac{1}{\e}\int_0^t \intt \nabla W \star (\rho - \rho^\e) \cdot \rho^\e(u^\e - u)  \,dxd\tau\cr
&=\frac{1}{\e}\int_0^t \intt \nabla W \star (\rho - \rho^\e) \cdot u (\rho - \rho^\e) \,dxd\tau - \frac1\e\int_0^t \intt \nabla W \star (\rho - \rho^\e) \cdot ( \rho u - \rho^\e u^\e) \,dxd\tau,
\end{aligned}
\end{align}
where the second term on the right hand side of the above equality can be rewritten as
$$\begin{aligned}
&-\frac{1}{2\e}\int_0^t \lt(\frac{d}{d\tau}\intt |\nabla W \star (\rho - \rho^\e)|^2\,dx\rt)d\tau\cr
&\quad =-\frac{1}{2\e}\intt |\nabla W \star (\rho - \rho^\e)|^2\,dx + \frac{1}{2\e}\intt |\nabla W \star (\rho_0 - \rho^\e_0)|^2\,dx
\end{aligned}$$
due to Lemma \ref{lem_wd}. On the other hand, by using the integration by parts, the first term on the right hand side of \eqref{I_3} can be estimated as
$$\begin{aligned}
&\frac1\e\int_0^t \intt \nabla W \star (\rho - \rho^\e) \cdot u (\rho - \rho^\e)\,dxd\tau \cr
&\quad = - \frac1\e\int_0^t\intt \nabla W \star (\rho - \rho^\e) \cdot u \lt(\Delta W \star (\rho - \rho^\e)\rt)\,dxd\tau\cr
&\quad = -\frac1{2\e}\int_0^t\intt |\nabla W \star (\rho - \rho^\e)|^2 \nabla \cdot u\,dxd\tau +\frac1\e\int_0^t \intt \nabla W \star (\rho - \rho^\e)\otimes \nabla W \star (\rho - \rho^\e) : \nabla u\,dxd\tau.
\end{aligned}$$
Combining all of the above observations implies
$$\begin{aligned}
I_3 &\leq -\frac{1}{2\e}\intt |\nabla W \star (\rho - \rho^\e)|^2\,dx + \frac{1}{2\e}\intt |\nabla W \star (\rho_0 - \rho^\e_0)|^2\,dx + \frac{3}{2\e}\|\nabla u\|_{L^\infty}\int_0^t \intt |\nabla W \star (\rho - \rho^\e)|^2\,dxd\tau.
\end{aligned}$$
This together with the estimates in the proof of Theorem \ref{main_thm} asserts
\begin{align}\label{est_re1}
\begin{aligned}
&\frac12\intt \rho^\e |u^\e - u|^2\,dx + \frac{1}{2\e}\intt |\nabla W \star (\rho - \rho^\e)|^2\,dx + \frac{(\gamma-C)}{\e} \int_0^t \intt \rho^\e |u^\e - u|^2\,dx d\tau\cr
&\quad + \frac1{2\e}\int_0^t\inttt \phi(x-y) \rho^\e(x) \rho^\e(y) |(u^\e(x) - u(x)) - (u^\e(y) - u(y)) |^2\,dxdyd\tau\cr
&\qquad \leq \frac12\intt \rho^\e_0 |u^\e_0 - u_0|^2\,dx +  \frac{1}{2\e}\intt |\nabla W \star (\rho_0 - \rho^\e_0)|^2\,dx  + \frac C\e \int_0^t \intt |\nabla W \star (\rho - \rho^\e)|^2\,dx d\tau + C\e
\end{aligned}
\end{align}
for $\e < 1$, where $C>0$ is independent of $\e>0$. This gives
$$\begin{aligned}
\intt |\nabla W \star (\rho - \rho^\e)|^2\,dx &\leq \e\intt \rho^\e_0 |u^\e_0 - u_0|^2\,dx +  \intt |\nabla W \star (\rho_0 - \rho^\e_0)|^2\,dx \cr
&\quad + C \int_0^t \intt |\nabla W \star (\rho - \rho^\e)|^2\,dx d\tau + C\e^2
\end{aligned}$$
for $\gamma > 0$ sufficiently large, and by applying Gr\"onwall's lemma to the above we obtain
\[
\intt |\nabla W \star (\rho - \rho^\e)|^2\,dx \leq C\e\intt \rho^\e_0 |u^\e_0 - u_0|^2\,dx +  C\intt |\nabla W \star (\rho_0 - \rho^\e_0)|^2\,dx + C\e^2
\]
for $\e < 1$, where $C>0$ is independent of $\e>0$. We again put this into \eqref{est_re1} to conclude
$$\begin{aligned}
&\frac12\intt \rho^\e |u^\e - u|^2\,dx + \frac{1}{2\e}\intt |\nabla W \star (\rho - \rho^\e)|^2\,dx + \frac{(\gamma-C)}{\e} \int_0^t \intt \rho^\e |u^\e - u|^2\,dx d\tau\cr
&\quad + \frac1{2\e}\int_0^t\inttt \phi(x-y) \rho^\e(x) \rho^\e(y) |(u^\e(x) - u(x)) - (u^\e(y) - u(y)) |^2\,dxdyd\tau\cr
&\qquad \leq C\intt \rho^\e_0 |u^\e_0 - u_0|^2\,dx +  \frac{C}{\e}\intt |\nabla W \star (\rho_0 - \rho^\e_0)|^2\,dx + C\e
\end{aligned}$$
for $\e < 1$, where $C>0$ is independent of $\e>0$. This completes the proof.
\end{proof}

%
%
%
%
\section*{Acknowledgments}
The author was supported by National Research Foundation of Korea(NRF) grant funded by the Korea government(MSIP) (No. 2017R1C1B2012918), POSCO Science Fellowship of POSCO TJ Park Foundation, and Yonsei University Research Fund of 2019-22-021.

%
%
%
%

\appendix

\section{Remarks on the regularity assumptions on $u$}\label{app_a}

Since $\rho > 0$, it follows from \eqref{main_eq22} that
\begin{align}\label{eqn_u}
\begin{aligned}
\gamma u(x,t) &= - (\nabla W \star \rho)(x,t) + \intt \phi(x-y)(u(y,t) - u(x,t))\rho(y,t)\,dy \cr
&= - (\nabla W \star \rho)(x,t) + (\phi \star (\rho u))(x,t) -   u(x,t) (\phi \star \rho)(x,t).
\end{aligned}
\end{align}
We first start with the estimate of $\|u\|_{L^\infty}$. From \eqref{eqn_u}, we easily get
\begin{align*}
\begin{aligned}
\lt|u(x,t)\rt| &= \lt|\frac{1}{\gamma +(\phi \star \rho)(x,t)}\lt(- (\nabla W \star \rho)(x,t) + (\phi \star (\rho u))(x,t)  \rt)\rt|\cr 
&\leq \frac1\gamma\lt(\|\nabla W \star \rho\|_{L^\infty} + \|\phi\|_{L^\infty}\|\rho\|_{L^1}\|u\|_{L^\infty} \rt),
\end{aligned}
\end{align*}
and subsequently this implies
\bq\label{est_u1}
\|u\|_{L^\infty} \leq \frac{\|\nabla W \star \rho\|_{L^\infty}}{\gamma - \|\phi\|_{L^\infty}\|\rho\|_{L^1}}
\eq
for $\gamma > 0$ large enough. We next estimate $\|\nabla u\|_{L^\infty}$. Taking the differential operator $\pa_{x_j}$ to \eqref{eqn_u} gives
\begin{align*}
\begin{aligned}
\pa_{x_j} u_i &= \frac{1}{\gamma + \phi \star \rho} \lt(-\pa_{x_j}\pa_{x_i} W \star \rho + \pa_{x_j}\phi \star (\rho u_i) \rt) - \frac{1}{(\gamma + \phi \star \rho)^2} \lt( \pa_{x_i} W \star \rho + \phi \star (\rho u_i)\rt) (\pa_{x_j} \phi \star \rho).
\end{aligned}
\end{align*}
This together with \eqref{est_u1} yields
\begin{align*}
\begin{aligned}
\|\nabla u\|_{L^\infty} 
&\leq \frac1\gamma \lt( \|\nabla(\nabla W \star \rho)\|_{L^\infty} + \|\nabla \phi\|_{L^\infty}\|\rho u\|_{L^1}\rt) + \frac{1}{\gamma^2} \lt( \|\nabla W \star \rho\|_{L^\infty}+ \|\phi\|_{L^\infty}\|\rho u\|_{L^1}\rt)\|\nabla \phi\|_{L^\infty}\|\rho\|_{L^1}\cr
&\leq \lt(1 + \frac{\|\nabla \phi\|_{L^\infty}\|\rho\|_{L^1}}{\gamma} \rt) \frac{1}{\gamma}\lt(\|\nabla W \star \rho\|_{\W^{1,\infty}} + \|\rho\|_{L^1}\|\phi\|_{\W^{1,\infty}}\|u\|_{L^\infty} \rt)\cr
&\leq \lt(1 + \frac{\|\nabla \phi\|_{L^\infty}\|\rho\|_{L^1}}{\gamma} \rt) \frac{1}{\gamma}  \lt(\|\nabla W\star\rho\|_{\W^{1,\infty}} + \frac{\|\rho\|_{L^1}\|\phi\|_{\W^{1,\infty}}\|\nabla W \star\rho\|_{L^\infty}}{\gamma - \|\phi\|_{L^\infty}\|\rho\|_{L^1}} \rt).
\end{aligned}
\end{align*}
We finally estimate $\|\pa_t u\|_{L^\infty}$. Similarly as before, we differentiate the equation \eqref{eqn_u} with respect to time $t$ to find
\begin{align}\label{est_ut}
\begin{aligned}
\pa_t u &= \frac{1}{\gamma + \phi \star \rho} \lt( -\nabla W \star \pa_t \rho + \phi \star (\pa_t (\rho u))\rt) -\frac{1}{(\gamma + \phi \star \rho)^2} \lt( -\nabla W \star \rho + \phi \star (\rho u)\rt)\lt( \phi \star \pa_t \rho\rt).
\end{aligned}
\end{align}
Here by using the continuity equation \eqref{main_eq2} we can estimate
\[
|\nabla W \star \pa_t \rho| = |\nabla^2 W \star (\rho u)| \leq \|\nabla(\nabla W \star\rho)\|_{L^\infty}\|u\|_{L^\infty}
\]
and
\[
|\phi \star \pa_t \rho| \leq \|\nabla \phi\|_{L^\infty}\|u\|_{L^\infty}\|\rho\|_{L^1}.
\]
Moreover, we also obtain
\begin{align*}
\begin{aligned}
|\phi \star (u\pa_t \rho)| &\leq \intt \rho(y)|u(y)\cdot \nabla \phi(x-y)| |u(y)| \,dy  + \intt \rho(y)|u(y)\cdot \nabla u(y)|\phi(x-y) \,dy\cr
&\leq \|u\|_{L^\infty}^2 \|\nabla \phi\|_{L^\infty}\|\rho\|_{L^1} + \|\phi\|_{L^\infty}\|\nabla u\|_{L^\infty}\|u\|_{L^\infty}\|\rho\|_{L^1}\cr
&\leq 2\|\phi\|_{\W^{1,\infty}}\|u\|_{\W^{1,\infty}}^2\|\rho\|_{L^1},
\end{aligned}
\end{align*}
which allows us to estimate
\begin{align*}
\begin{aligned}
|\phi \star (\pa_t (\rho u))| &= |\phi \star (u \pa_t\rho) + \phi \star (\rho\pa_t u)| \leq 2\|\phi\|_{\W^{1,\infty}}\|u\|_{\W^{1,\infty}}^2\|\rho\|_{L^1} + \|\phi\|_{L^\infty}\|\rho\|_{L^1}\|\pa_t u\|_{L^\infty}.
\end{aligned}
\end{align*}
Then we now combine all of the above estimates with \eqref{est_ut} to have
\begin{align*}
\begin{aligned}
\|\pa_t u\|_{L^\infty} &\leq \frac1\gamma \lt(\|\nabla W \star \pa_t \rho\|_{L^\infty} + \|\phi \star (\pa_t (\rho u))\|_{L^\infty} \rt) + \frac{1}{\gamma^2} \lt(\|\nabla W \star \rho\|_{L^\infty} + \|\phi \star (\rho u)\|_{L^\infty} \rt) \|\phi \star \pa_t \rho\|_{L^\infty}\cr
&\leq \frac1\gamma \lt( \|\nabla(\nabla W \star\rho)\|_{L^\infty}\|u\|_{L^\infty} + 2\|\phi\|_{\W^{1,\infty}}\|u\|_{\W^{1,\infty}}^2 \|\rho\|_{L^1} \rt)  + \frac1\gamma\|\phi\|_{L^\infty}\|\rho\|_{L^1}\|\pa_t u\|_{L^\infty} \cr
&\quad + \frac{1}{\gamma^2}\lt(\|\nabla W \star \rho\|_{L^\infty} + \|\phi\|_{L^\infty}\|u\|_{L^\infty}\|\rho\|_{L^1} \rt) \|\nabla\phi\|_{L^\infty}\|u\|_{L^\infty}\|\rho\|_{L^1}.
\end{aligned}
\end{align*}
Hence we finally have
\begin{align*}
\begin{aligned}
\|\pa_t u\|_{L^\infty} &\leq \frac{1}{\gamma - \|\phi\|_{L^\infty}\|\rho\|_{L^1}}\lt( \|\nabla(\nabla W \star\rho)\|_{L^\infty}\|u\|_{L^\infty} + 2\|\rho\|_{L^1}\|\phi\|_{\W^{1,\infty}}\|u\|_{\W^{1,\infty}}^2  \rt)\cr
&\quad + \frac{\|\rho\|_{L^1}\|\nabla \phi\|_{L^\infty}\|u\|_{L^\infty}}{\gamma(\gamma - \|\phi\|_{L^\infty}\|\rho\|_{L^1})}\lt(\|\nabla W \star\rho\|_{L^\infty} + \|\rho\|_{L^1}\|\phi\|_{L^\infty}\|u\|_{L^\infty} \rt).
\end{aligned}
\end{align*}
As observed above, since $\|u\|_{\W^{1,\infty}}$ can be bounded from above by some constant which depends only on $\|\nabla W \star \rho\|_{\W^{1,\infty}}$, $\|\phi\|_{\W^{1,\infty}}$, $\|\rho\|_{L^1}$, and $\gamma$, so does $\|\pa_t u\|_{L^\infty}$.

%
%
%
%


\begin{thebibliography}{10}

\bibitem{AGS08} L. Ambrosio, N. Gigli, and G, Savar\'e, Gradient flows: in metric spaces and in the space of probability measures, Springer Science \& Business Media, 2008.

\bibitem{CCpre} J. A. Carrillo, Y.-P. Choi, Quantitative error estimates for the large friction limit of Vlasov equation with nonlocal forces, Ann. Inst. H. Poincar\'e Anal. Non Lin\'eaire, 37, (2020), 925--954.

\bibitem{CCHS19} J. A. Carrillo, Y.-P. Choi, M. Hauray, and S. Salem, Mean-field limit for collective behavior models with sharp sensitivity regions, J. Eur. Math. Soc., 21, (2019), 121--161.

\bibitem{CCJpre} J. A. Carrillo, Y.-P. Choi, and J. Jung, Quantifying the hydrodynamic limit of Vlasov-type equations with alignment and nonlocal forces, preprint.

\bibitem{CCK16} J. A. Carrillo, Y.-P. Choi, and T. K. Karper, On the analysis of a coupled kinetic-fluid model with local alignment forces, Ann. I. H. Poincar\'e - AN., 33, (2016), 273--307.

\bibitem{CCP17} J. A. Carrillo, Y.-P. Choi, and S. P\'erez, A review on attractive-repulsive hydrodynamics for consensus in collective behavior. Active particles. Vol. 1. Advances in theory, models, and applications, 259--298, Model. Simul. Sci. Eng. Technol., Birkh\"auser/Springer, Cham, 2017.

\bibitem{CCTT16} J. A. Carrillo, Y.-P. Choi, E. Tadmor, and C. Tan, Critical thresholds in 1D Euler equations with nonlocal forces, Math. Mod. Methods Appl. Sci., 26, (2016), 85--206.

\bibitem{CCT19} J. A. Carrillo, Y.-P. Choi, O. Tse, Convergence to equilibrium in Wasserstein distance for damped Euler equations with interaction forces, Comm. Math. Phys., 365, (2019), 329--361.

\bibitem{CCZ16}  J. A. Carrillo, Y.-P. Choi and E. Zatorska, On the pressureless damped Euler-Poisson equations with quadratic confinement: critical thresholds and large-time behavior, Math. Mod. Methods Appl. Sci., 26, (2016), 2311-2340.

\bibitem{CFGS17} J. A. Carrillo, E. Feireisl, P. Gwiazda and A. \'Swierczewska-Gwiazda, Weak solutions for Euler systems with non-local interactions, J. London Math. Soc., 95, (2017), 705--724.

\bibitem{CPW20} J. A. Carrillo, Y. Peng, and A. Wr\'oblewska-Kami\'nska, Relative entropy method for the relaxation limit of hydrodynamic models, Netw. Heterog. Media, 15, (2020), 369--387.


\bibitem{C19} Y.-P. Choi, The global Cauchy problem for compressible Euler equations with a nonlocal dissipation, Math. Mod. Methods Appl. Sci., 29, (2019), 185--207.



\bibitem{Cpre} Y.-P. Choi, A rigorous derivation of the hydrodynamic model for synchronization phenomena from inertial kinetic Kuramoto model, preprint.

\bibitem{CHL17} Y.-P. Choi, S.-Y. Ha and Z. Li, Emergent dynamics of the Cucker-Smale flocking model and its variants. Active particles. Vol. 1. Advances in theory, models, and applications, 299--331, Model. Simul. Sci. Eng. Technol., Birkh\"auser/Springer, Cham, 2017.

\bibitem{CH19} Y.-P. Choi and J. Haskovec, Hydrodynamic Cucker-Smale model with normalized communication weights and time delay, SIAM J. Math. Anal., 51, (2019), 2660--2685.

\bibitem{CJpre} Y.-P. Choi and I.-J. Jeong, On well-posedness and singularity formation for the Euler-Riesz system, preprint.

\bibitem{CYpre} Y.-P. Choi and S.-B. Yun, Existence and hydrodynamic limit for a Paveri-Fontana type kinetic traffic model, preprint.

\bibitem{CG07} J.-F. Coulombel and T.  Goudon, The strong relaxation limit of the multidimensional isothermal Euler equations. Trans. Am. Math. Soc.,  359, (2007), 637--648.

\bibitem{CS07} F. Cucker and S. Smale, Emergent behavior in flocks, IEEE Trans. Autom. Control, 52, (2007) 852.

\bibitem{Daf79} C. M. Dafermos: The second law of thermodynamics and stability, Arch. Ration. Mech. Anal., 70, (1979), 167--179.

\bibitem{DS10} C. DeLellis and L. Sz\'ekelyhidi, On admissibility criteria for weak solutions of the Euler equations, Arch. Ration. Mech. Anal., 195, (2010), 225--260.

\bibitem{FS15} R. Fetecau and W. Sun, First-order aggregation models and zero inertia limits, J. Differential Equations, 259, (2015), 6774--6802.

\bibitem{FST16} R. C. Fetecau, W. Sun, and C. Tan, First-order aggregation models with alignment, Physica D, 325, (2016), 146--163.

\bibitem{FK19} A. Figalli and M.-J. Kang, A rigorous derivation from the kinetic Cucker-Smale model to the pressureless Euler system with nonlocal alignment, Anal. PDE, 12, (2019), 843--866.

\bibitem{HKK15} S.-Y. Ha, M.-J. Kang, and B. Kwon, Emergent dynamics for the hydrodynamic Cucker-Smale system in a moving domain, SIAM J. Math. Anal., 47, (2015), 3813--3831.


\bibitem{HL09} S.-Y. Ha and J.-G. Liu, A simple proof of the Cucker-Smale flocking dynamics and mean-field limit, Comm. Math. Sci., 7, (2009), 297--325.

\bibitem{HT08} S.-Y. Ha and E. Tadmor, From particle to kinetic and hydrodynamic descriptions of flocking, Kinet. Relat. Models, 1, (2008), 415--435.


\bibitem{Jab00} P.-E. Jabin, Macroscopic limit of Vlasov type equations with friction, Ann. Inst. H. Poincar\'e Anal.  Non Lin\'eaire, 17, (2000), 651--672.


\bibitem{KMT15} T. Karper, A. Mellet and K. Trivisa, Hydrodynamic limit of the kinetic Cucker-Smale flocking model, Math. Models Methods Appl. Sci., 25, (2015), 131-163.

\bibitem{LT13} C. Lattanzio, A. E. Tzavaras, Relative entropy in diffusive relaxation, SIAM J. Math. Anal., 45, (2013), 1563--1584.

\bibitem{LT17} C. Lattanzio, A. E. Tzavaras, From gas dynamics with large friction to gradient flows describing diffusion theories, Comm. Partial Differential Equations, 42, (2017), 261--290.



\bibitem{TT14} E. Tadmor and C. Tan, Critical thresholds in flocking hydrodynamics with nonlocal alignment, Philos. Trans. A Math. Phys. Engrg. Sci., 372, (2014) 20130401.


\bibitem{Tik52} A.N. Tikhonov, Systems of differential equations containing small parameters in the derivatives, Mat. Sb. (NS) 31, (1952), 575--586.

\bibitem{Vill03} C. Villani, Topics in optimal transportation, Graduate Studies in Mathematics, 58. American Mathematical Society, Providence, RI, 2003.

\end{thebibliography}
\end{document}